\documentclass[11pt]{article}
\usepackage{amsmath, amsthm, amssymb, amsfonts, xspace, enumitem}
\usepackage{mathrsfs}
\usepackage[sort&compress,numbers,comma,square]{natbib}
\usepackage[active]{srcltx}
\usepackage{color}

\usepackage{environ}
\makeatletter
\NewEnviron{tagleft}{\tagsleft@true\BODY}  
\NewEnviron{tagright}{\tagsleft@false\BODY}
\NewEnviron{hide}{}                        
\makeatother

\newtheorem{theorem}{Theorem}
\newtheorem{cor}[theorem]{Corollary}
\newtheorem{prop}[theorem]{Proposition}
\newtheorem{lemma}[theorem]{Lemma}
\theoremstyle{remark}
\newtheorem*{remark}{Remark}

\def\Grp#1{\left(#1\right)}
\def\Cbr#1{\left\{#1\right\}}
\def\Sbr#1{\left[#1\right]}
\def\Flr#1{\left\lfloor#1\right\rfloor}

\def\Abs#1{\left|#1\right|}

\def\cf#1{\mathbf{1}\!\Cbr{#1}}
\def\dd{\mathrm{d}}
\def\gv{\,|\,}
\def\mean{\mathbb{E}}
\def\prob{\mathbb{P}}
\def\pr#1{\prob\{#1\}}

\def\rx{\epsilon}
\def\LT#1{\Cal L[{#1}]}
\def\FT#1{\widehat{#1}}
\def\Re{\mathrm{Re}}
\def\Im{\mathrm{Im}}

\def\nth#1{\frac{1}{#1}}
\def\lfrac#1#2{#1/#2}  

\def\levy{L\'evy\xspace}

\def\Coms{\mathbb{C}}
\def\Ints{\mathbb{Z}}
\def\Nats{\mathbb{N}}
\def\Reals{\mathbb{R}}

\makeatletter
\def\@cleandot{\@ifnextchar.{}{\@ifnextchar,{.}{\@ifnextchar;{.}{\@ifnextchar?{.}{\@ifnextchar:{.}{\@ifnextchar!{.}{\@ifnextchar'{.}{\@ifnextchar){.}{.\ }}}}}}}}}
\def\as{a.s\@cleandot}
\def\iid{i.i.d\@cleandot}
\def\pdf{p.d.f\@cleandot}
\def\rhs{r.h.s\@cleandot}
\def\wrt{w.r.t\@cleandot}
\makeatother

\def\rsup#1{\overline{#1}} 
\def\rinf#1{\underline{#1}} 
\def\dual#1{\widehat{#1}}
\def\levym{\Pi}  
\def\jp{\Delta}
\def\fpt{T} 
\def\fht{\tau} 

\def\da{/\alpha}     
\def\ra{{1\da}}  

\def\sumoi#1{\sum_{#1=1}^\infty}
\def\sumzi#1{\sum_{#1=0}^\infty}
\def\intii{\int_{-\infty}^\infty}
\def\intzi{\int_0^\infty}

\def\iunit{\mathrm{i}}

\def\Sp#1{\sp{(#1)}}

\def\toi{\to\infty}
\def\dto{\downarrow}
\def\uto{\uparrow}

\def\Cal#1{\mathcal{#1}}
\def\inum#1{#1_1, #1_2, \ldots}                 

\def\Dunif{\mathrm{Uniform}}
\def\Dbeta{\mathrm{Beta}}

\setlist{itemsep=0ex}
\newlist{thmenum}{enumerate}{1}
\setlist[thmenum]{label={(\alph*)},leftmargin=5ex}

\begin{document}
\begin{center}
  \large
  \textbf{Law of the first passage triple of a spectrally positive
    strictly stable process
  }  
  \\[1.5ex]
  \normalsize
  Zhiyi Chi \\
  Department of Statistics\\
  University of Connecticut \\
  Storrs, CT 06269, USA, \\[.5ex]
  E-mail: zhiyi.chi@uconn.edu \\[1ex]
  \today
\end{center}

\begin{abstract}
  For a spectrally positive and strictly stable process with index in
  (1,\ 2), a series representation is obtained for the joint
  distribution of the ``first passage triple'' that consists of the
  time of first passage and the undershoot and the overshoot at first
  passage.  The result leads to several corollaries, including
  1) the joint law of the first passage triple and the pre-passage
  running supremum, and 2) at a fixed time point, the joint law of the
  process' value, running supremum, and the time of the running
  supremum.  The representation can be decomposed as a sum of strictly
  positive functions that allows exact sampling of the first passage
  triple.

  \medbreak

  \medbreak

  \medbreak
  \textbf{Acknowledgment.}  The research is partially supported by NSF
  Grant DMS 1720218.  The author would like to thank two referees for
  their careful reading of the manuscript and useful suggestions,  in
  particular, one that significantly simplifies the proof of the main
  theorem.
\end{abstract}

\section{Introduction}
Let $X = (X_t)_{t\ge 0}$ be a \levy process and $\levym(\dd x)$ its
\levy measure.  Denote by
\begin{align*}
  \jp_t = X_t - X_{t-}, \quad
  \rsup X_t = \sup_{0\le s\le t} X_s,
\end{align*}
the jump and running supremum of $X$ at $t$, respectively.  By
convention, $X_{0-} = X_0 = 0$.  For $c\ge 0$, the first passage time
of $X$ at level $c$ is defined as
\begin{align*}
  \fpt_c = \inf\{t>0: X_t>c\},
\end{align*}
while for $x\in \Reals$, the first hitting time of $X$ at $x$ is
defined as
\begin{align*}
    \fht_x = \inf\{t>0: X_t = x\},
\end{align*}
where by convention $\inf\emptyset = \infty$.

By definition, a \levy process is spectrally positive if it only has
positive jumps, i.e.\ its \levy measure is concentrated on $(0,
\infty)$.  It is well known that if $X$ is spectrally positive and is
not a subordinator, then $(\fht_{-x})_{x\ge 0}$ is a subordinator,
possibly killed at an exponential time and for $t$, $x>0$, $t \pr{
  \fht_{-x} \in\dd t}\,\dd x = x \pr{ X_t\in -\dd x}\,\dd t$, which is
known as Kendall's identity (\cite {bertoin:96:cup}, Chapter VII).  If
for each $t>0$, $X_t$ has a probability density function (\pdf)
$g_t(x)$, then for each $x>0$, $\fht_{-x}$ has a \pdf $f_{-x}(t)$ and
Kendall's identity can be written as
\begin{align}  \label{e:Kendall}
  t f_{-x}(t) = x g_t(-x), \quad x>0,\ t>0.
\end{align}
In this paper, a \pdf is always defined with respect to (\wrt) the
Lebesgue measure.

Let $X$ be a spectrally positive and strictly stable process with
index $\alpha\in (1,2)$.  The first passage of $X$ at a fixed level
$c>0$ is of particular interest and has already drawn a lot of
attention.  The joint distribution of $X_{\fpt_c-}$ and $\jp
_{\fpt_c}$ is known \cite{doney:06:aap} and so is the distribution of
$\fpt_c$ \cite{bernyk:08:ap, simon:10:expomath, simon:11:sto}.
Related to these random variables, the distribution of $\fht_x$ is
classical when $x<0$ \cite{bertoin:96:cup} and is also known when
$x>0$ \cite{peskir:08:ecp, simon:10:expomath}.  On the other hand, the
three random variables $\fpt_c$, $X_{\fpt_c-}$, and $\jp_{\fpt_c}$
completely describes what happens to $X$ at the moment of first
passage.  Although some general results are available \cite
{doney:06:aap}, explicit representations of the joint distribution of
the triple have been unknown.

While there may be many different representations, those that allow
exact sampling are practically more useful and perhaps conceptually
more satisfactory.  Ideally, a representation should also allow
efficient implementation of the sampling.  Although such
representations are available for the marginal distributions of $X_t$,
$\rsup X_t$,  $\fpt_c$, and $\fht_x$ \cite {zolotarev:66:stmsp,
  simon:10:expomath, simon:11:sto}, they seem much harder to get for
the joint distribution of $\fpt_c$, $X_{\fpt_c-}$, and $\jp_{\fpt_c}$,
so we will content ourselves with a representation that allows exact
sampling of the triple regardless of efficiency.

The following function will play an important role.  For $c>0$, $x\in
(-\infty, c)$, and $t>0$, define
\begin{align} \label{e:hc}
  h_c(x,t) = \frac{\pr{X_t\in\dd x, \rsup X_t\le c}}{\dd x}.
\end{align}
Since $X$ has the scaling property, i.e.\ $(X_{\lambda t})_{t\ge  0}
\sim  (\lambda^\ra X_t)_{t\ge 0}$ for all $\lambda >0$, one can assume
without loss of generality that 
\begin{align} \label{e:stable-normal}
  \mean (e^{-q X_t}) = \exp(t q^\alpha), \quad t>0,\ q\ge0.
\end{align}
Because also by scaling
\begin{align} \label{e:scaling}
  (\fpt_c, X_{\fpt_c-}, \jp_{\fpt_c})
  \sim
  (c^\alpha \fpt_1, c X_{\fpt_1-}, c\jp_{\fpt_1}),
\end{align}
it suffices to consider $c=1$.

\begin{theorem} \label{t:h}
  Suppose $X$ is a stable process with index $\alpha\in (1,2)$
  satisfying \eqref{e:stable-normal}.  Then the triple $(\fpt_1,
  X_{\fpt_1-}, \jp_{\fpt_1})$ has a \pdf that at each $(t, x, z)\in
  (0,\infty)\times \Reals \times [0,\infty)$ takes value
  \begin{align*}
    \varrho_1(t,x,z)=
    \frac{z^{-\alpha-1}}{\Gamma(-\alpha)}\cf{x<1<x+z}
    h_1(x,t),
  \end{align*}
  where for $x\in (-\infty, 1)$,
  \begin{align} \label{e:h-power}
    h_1(x,t)
    &=
      \nth{\pi}\sumoi{k,n}(-1)^{k+n}
      \frac{\Gamma(k\da + n)}{\Gamma(\alpha n) k!}
      \sin(\pi k\da) (1-x)^k t^{-k\da - n}.
  \end{align}
  The series in \eqref{e:h-power} converges absolutely for given
  $x$ and $t>0$.
\end{theorem}

Given $c>0$, by the scaling relation \eqref{e:scaling}, $(\fpt_c,  X_{\fpt_c-}, \jp_{\fpt_c})$ has joint \pdf
\begin{align*}
  \varrho_c(t,x,z) =  c^{-\alpha-2}
  \varrho_1(c^{-\alpha} t, c^{-1} x, c^{-1} z).
\end{align*}
Furthermore,
\begin{align}  \label{e:h-scaling}
  h_c(x,t) = c^{-1} h_1(c^{-1} x, c^{-\alpha} t).
\end{align}
The core of Theorem \ref{t:h} is
\eqref{e:h-power} and a key step in its proof is to show
\begin{align}\label{e:h-series}
  h_1(x,t)=
  \sumzi n \frac{f\Sp n_{x-1}(t)}{\Gamma(\alpha n + \alpha)},
\end{align}
which can be formally written as 
\begin{align*}
  h_1(x,\cdot) = E_{\alpha, \alpha}(D) f_{x-1},
\end{align*}
where $D$ is the differential operator and $E_{\alpha,\alpha}(s)$ is a
Mittag-Leffler function (\cite{NIST:10, erdelyi:55:mcgraw}; see
section \ref{ss:laplace}).  Many detailed asymptotics of $f\Sp
n_{x-1}(t)$ can be found in \cite{gawronski:88:ap}.  It will be seen
that conditionally on $X_{\fpt_1-}=x$, $\jp_{\fpt_1}$ and $\fpt_1$ are
independent, with the latter having \pdf $h_1(x,\cdot)/v_1(x)$, where
\begin{align} \label{e:v-stable}
  v_1(x)
  =\intzi h_1(x,t)\,\dd t
  =
  \frac{1 - (x\vee0)^{\alpha-1}}{\Gamma(\alpha)}.
\end{align}
One may have noticed that when $x\in (0,1)$, $v_1(x)$ is strictly
smaller than $1/\Gamma(\alpha)$, whereas the sum of the term-wise
integrals of the series \eqref{e:h-series} is $1/\Gamma(\alpha)$.  The
lack of interchangeability of summation and integration reflects the
high oscillations of $f\Sp n_{x-1}(t)$ as functions of $t$, which
are tricky to tackle directly.  In this paper, \eqref {e:h-series}
will be first established for $x<a$, where $a\le0$ is a certain
constant, and then it will be established for all $x<1$ by analytic
extension.

Several results can be derived from Theorem \ref{t:h}.  First, an
integral representation of $h_1(x,t)$.
\begin{cor} \label{c:h}
  Under the same condition as above,
  \begin{align*}
    h_1(x,t)
    =
    \nth{\pi} 
    \intzi e^{-s t -(1-x) s^\ra \cos(\pi\da)}
    \sin((1-x) s^\ra \sin(\pi\da)) E_{\alpha, \alpha}(-s)\,\dd s.
  \end{align*}
\end{cor}

The next result on the support of $h_1(x,t)$ will be used later and is
of interest in its own right.
\begin{cor} \label{c:nonzero}
  $h_1(x,t)>0$ for all $x<1$ and $t>0$.
\end{cor}

In the last two corollaries, $h_c(x,t)$ is regarded as a function of
$t$ and $x$ with $c=1$ being fixed.  When $t$ is fixed and $c$ and $x$
are treated as variables, $h_c(x,t)$ provides the joint distribution
of $X_t$ and $\rsup X_t$.  Specifically, from \eqref{e:h-power} and
scaling, the following result obtains.  Since $(X_t, \rsup X_t)\sim
(t^\ra X_1, t^\ra \rsup X_1)$, it suffices to consider $t=1$.
\begin{cor} \label{c:joint-max}
  $X_1$ and $\rsup X_1$ have joint \pdf
  \begin{align*}
    \frac{\pr{X_1\in\dd x, \rsup X_1\in\dd c}}{\dd x\,\dd c}
    =
    \cf{c>(x\vee0)} \frac{\partial h_c(x,1)}{\partial c}
  \end{align*}
  with
  \begin{align}
    \frac{\partial h_c(x,1)}{\partial c}
    &=
    \nth{\pi}
    \sumoi{k,n} \frac{\Gamma(k\da + n)}{\Gamma(\alpha n) k!}
    (-1)^{k+n}  \sin(\pi k\da) [k c + (\alpha n-1)(c-x)]
    (c-x)^{k-1} c^{\alpha n -2}.
    \label{e:h-diff}
  \end{align}
\end{cor}
\begin{remark}
  For a standard Brownian motion $W$, it is known that 
  (\cite{jeanblanc:09:sv-l}, Corollary 3.2.1.2).
  \begin{align} \label{e:Brownian}
    \frac{\pr{W_1\in\dd x, \sup_{s\le 1} W_s\in\dd c}}{\dd
      x\,\dd c}
    =
    \cf{c>(x\vee0)} \frac{2(2c-x)}{\sqrt{2\pi}}
    \Cbr{-\frac{(2c-x)^2}{2}}.
  \end{align}
  It will be shown in the Appendix that \eqref{e:Brownian} can be
  deduced from \eqref{e:h-diff}.  Note that by \eqref
  {e:stable-normal}, for $\alpha=2$, $(X_t)_{t\ge 0}
  \sim(W_{2t})_{t\ge 0}$.
\end{remark}

The next corollary combined with Theorem \ref{t:h} gives the joint
distribution of $\fpt_1$, $X_{\fpt_1-}$, $\jp_{\fpt_1}$, and the
pre-passage running supremum $\rsup X_{\fpt_1-}$.
\begin{cor}\label{c:pre-max}
  Conditionally on $\fpt_1=t$ and $X_{\fpt_1-}=x<1$, $\jp_{\fpt_1}$
  and $\rsup X_{\fpt_1-}$ are independent, such that $\jp_{\fpt_1}$
  follows a Pareto distribution with
  \begin{align*}
    \pr{\jp_{\fpt_1}\in
    \dd z\gv \fpt_1=t, X_{\fpt_1-}=x} = \alpha (1-x)^\alpha
    z^{-\alpha-1}\cf{z > 1-x}\,\dd z,
  \end{align*}
  and for each $c\in [x\vee 0, 1]$,
  \begin{align*}
    \pr{\rsup X_{\fpt_1-} \le c\gv\fpt_1=t, X_{\fpt_1-}=x} =
    h_c(x,t)/h_1(x,t).
  \end{align*}
\end{cor}
\begin{remark}
  The foundation of conditional probability and conditional \pdf is
  measure theory \cite {breiman:92:siam}.  In Corollary \ref
  {c:pre-max}, each can be expressed in terms of a joint \pdf.  For
  example, if $k(z,t,x)$ is the joint \pdf of $\rsup X_{\fpt_1-}$,
  $\fpt_1$, and $X_{\fpt_1-}$, then $\pr{\rsup X_{\fpt_1-} \le
    c\gv\fpt_1=t, X_{\fpt_1-}=x}=\int^c_0 k(z, t, x)\,\dd z/\int^1_0
  k(z, t,x)\,\dd z$.
\end{remark}

By further analysis of $h_c(x,t)$, the joint \pdf of $X_t$, the
running supremum $\rsup X_t$, and the time of the running supremum
$\rsup G_t = \sup\{s<t: X_s = \rsup X_s\}$ can be obtained.  Since by
scaling
\begin{align*}
  (\rsup G_t, \rsup X_t, X_t)\sim (t\rsup G_1, t^\ra\rsup X_1, t^\ra
  X_1),
\end{align*}
it suffices to consider $t=1$.  As noted earlier, the distribution of
$\rsup X_1$ is known \cite {bernyk:08:ap, simon:10:expomath,
  simon:11:sto}.  The distribution of $\rsup G_1$ is also known.
Indeed, $\rsup G_t = \Lambda_{\vartheta_t-}$, where $\vartheta_t =
\inf\{s>0: \Lambda_s>t\}$ and $\Lambda$ is the ladder time process of
$X$, which is strictly stable with index $1-\ra$ (\cite
{bertoin:96:cup}, Lemma VIII.1).  Then by scaling, $\rsup G_1 \sim
\rsup G_t/t = \Lambda_{\vartheta_t-}/t$ and letting $t\to0$ yields
$\rsup G_1\sim \Dbeta(1-\ra, \ra)$ according to the generalized
arcsine law (\cite {bertoin:96:cup}, Theorem III.6).  That is, the
\pdf of $\rsup G_1$ at $x\in (0,1)$ is $\pi^{-1} \sin(\pi\da) x^{-\ra}
(1-x)^{\ra-1}$.  Also, from the excursion theory (\cite
{bertoin:96:cup},  IV.4), conditionally on $\rsup G_1$, $(X_t)_{t\le
  \rsup G_1}$ and $(X_{t+\rsup G_1} - X_{\rsup G_1})_{t\le 1-\rsup
  G_1}$ are independent.  With this background, we have the next
result.  By $(\rinf G_1, \rinf X_1, X_1)\sim (1-\rsup G_1, X_1 - \rsup
X_1, X_1)$, where $\rinf X_t = \inf_{0\le s\le t} X_s$ and $\rinf G_t
= \sup\{s<t: X_s = \rinf X_s\}$, it also provides the joint \pdf of
$\rinf G_1$, $\rinf X_1$, and $X_1$.
\begin{cor}\label{c:running-max}
  $\rsup G_1$, $\rsup X_1$, and $X_1$ have joint \pdf
  \begin{align} \label{e:running-max}
    \frac{\pr{\rsup G_1\in\dd r, \rsup X_1\in\dd c, X_1\in\dd x}}
    {\dd r\,\dd c\,\dd x}
    =
    m(c,r) f_{x-c}(1-r)
  \end{align}
  for $r\in (0,1)$ and $c>x\vee 0$, where 
  \begin{align} \label{e:hc-y1}
    m(c,r)
    &=
    \frac{\sin(\pi\da)}{\pi c} \sumoi n
    \frac{\Gamma(\ra + n)}{\Gamma(\alpha n)}
    (-1)^{1+n} c^{\alpha n} r^{-\ra-n}
    \\\label{e:hc-y2}
    &=
    \frac{\sin(\pi\da)}{\pi c^2}
    \intzi s^\ra E_{\alpha,\alpha}(-s) e^{-s r/c^\alpha}\,\dd s > 0.
  \end{align}
  Moreover, conditionally on $\rsup G_1=r\in (0,t)$, $\rsup X_1$ and
  $\rsup X_1 - X_1$ are independent, such that
  \begin{align} \label{e:max-m}
    \frac{\pr{\rsup X_1\in\dd c\gv \rsup G_1=r}}{\dd c}
    = \Gamma(1- \ra) r^\ra m(c,r), \quad c>0,
  \end{align}
  and
  \begin{align*}
    \frac{\pr{\rsup X_1-X_1\in \dd x\gv \rsup G_1 =r}}{\dd x}
    =
    \frac{\Gamma(\ra) x g_{1-r}(-x)}{(1-r)^\ra}, \quad x>0.
  \end{align*}
\end{cor}

\begin{remark}
  1)  For any $r>0$,
  \begin{align*}
    \frac{\pr{\rsup X_1\in r^\ra \dd c\gv \rsup G_1=r}}{\dd c}
    = \Gamma(1-\ra) r^{2\da} m(r^\ra c, r)
    = \Gamma(1-\ra) m(c,1),
  \end{align*}
  therefore $\rsup X_1/\rsup G_1^\ra$ is independent of $\rsup
  G_1$.  This is a special case of the result in \cite
  {molchan:01:tpa} that shows the independence for any strictly stable
  process.

  2) By duality, it is natural to interpret $\Gamma(\ra) x g_1(-x) =
  \Gamma(\ra) f_{-x}(1)$, $x>0$, as the conditional \pdf of $X_1$ at
  $-x$ given $\rsup X_1=0$.  Likewise, by letting $r=1$ in \eqref
  {e:max-m} and reading $\pr{\rsup X_1\in\dd c\gv \rsup G_1=1}$
  as $\pr{\rsup X_1\in\dd c\gv X_1 = \rsup X_1} = \pr{\rsup X_1\in\dd
    c\gv \rinf X_1=0}$, $\Gamma(1-\ra) m(c,1)$ may be interpreted as
  the conditional \pdf of $\rsup X_1$ at $c$ given $\rinf X_1=0$; see
  more comments in section \ref{s:proof}.
  
  3) It is worth mentioning that, for a \levy process $X$ in general,
  if under its law 0 is regular for $(0,\infty)$ and for $(-\infty,
  0)$, then for any $t>0$, $X$ is continuous at $\rsup G_t$.  First,
  $\rsup G_t\in(0,t)$ \as (see \cite{bertoin:96:cup}, p.~157).
  Second, given $\rx>0$, any $t_0\in (0,t)$ where $X$ makes a positive
  jump of size at least $\rx$ is a stopping time, so by the regularity
  of 0 for $(0,\infty)$, there are infinitely many $1>t_n\dto t_0$
  with $X_{t_n}>X_{t_0} > X_{t_0-}$.  On the other hand, any $t_0\in
  (0,t)$ where $X$ makes a negative jump of absolute size at least
  $\rx$ is a stopping time, so by duality and the regularity of 0 for
  $(-\infty, 0)$, there are infinitely many $0<t_n\uto t_0$ with
  $X_{t_n}>X_{t_0-} > X_{t_0}$.  Since $\rx>0$ is arbitrary, this
  implies that $\rsup G_t$ cannot be a time where $X$ makes a jump,
  and so $X$ is continuous at $\rsup G_t$.
\end{remark}

It can be seen that $m(c,t)\,\dd t\,\dd c$ is the renewal measure of
the bivariate ascending ladder (time and height) process of $X$, by
using the quintuple law for first passage in \cite {doney:06:aap} or
more directly, by using $\mean(e^{-\beta \rsup X_\tau}) = \kappa(q,
0)/\kappa(q, \beta) $, $q>0$, $\beta>0$, where $\kappa(\lambda,
\beta)$ is the characteristic exponent of the ladder process, and
$\tau$ is a random variable with \pdf $q e^{-q x} \cf{x>0}$
independent of $X$ (\cite {bertoin:96:cup}, p.~163).  First, by
\eqref{e:max-m} and $\rsup G_1\sim \Dbeta(1-\ra, \ra)$, the joint \pdf
of $(\rsup G_1, \rsup X_1)$ can be written down.  Then by scaling and
\eqref{e:hc-y2}, for each $t>0$, $(\rsup G_t, \rsup X_t)$ has joint
\pdf
\begin{align*}
  \frac{t^{-1 - \ra} m(ct^{-\ra}, r/t)(1-r/t)^{\ra-1} \cf{0<r<t}}
  {\Gamma(\ra)}
  =
  \frac{m(c, r)(t-r)^{\ra-1}\cf{0<r<t}}{\Gamma(\ra)}.
\end{align*}
Then
\begin{align*}
  \mean(e^{-\beta \rsup X_\tau})
  &=
    \frac{1}{\Gamma(\ra)}\int_{c>0, t>r>0}
    m(c,t)(t-r)^{\ra-1} e^{-\beta c}
    \,\dd r\,\dd c\times (q e^{-q t})\,\dd t
  \\
  &=
    q^{1-\ra} \int_{c>0,r>0} m(c,r) e^{-q r - \beta c}\,\dd r\,\dd c.
\end{align*}
On the other hand, $\kappa(q,0) = q^{1-\ra}$ (\cite{bertoin:96:cup},
p.~218).  Therefore,
\begin{align}  \label{e:ladder0}
  \kappa(q,\beta) = \Grp{\int_{c>0,r>0} m(c,r) e^{-q r - \beta c}\,\dd
  r\,\dd c}^{-1},
\end{align}
and so $m(c,r)$ is the density of the renewal measure of the
ladder process.  From \eqref{e:hc-y2},
\begin{align*}
  \intzi m(c,r) e^{-q r}\,\dd r
  =
  \frac{\sin(\pi\da)}{\pi c^2}
  \intzi \frac{s^\ra E_{\alpha,\alpha}(-s)}{q+s/c^\alpha}\,\dd s.
\end{align*}
The integral representation does not seem to provide an easy path to
an explicit formula for $\kappa(q,\beta)$.  On the other hand, it can
be shown that for $q\ge0$, $\beta\ge0$,
\begin{align} \label{e:ladder}
  \kappa(q,\beta) =
  \begin{cases}
    \displaystyle \frac{\beta^\alpha - q}{\beta - q^\ra} & \text{if
    } \beta \ne q^\ra, \\[2ex]
    \alpha \beta^{\alpha-1} & \text{else.}
  \end{cases}
\end{align}

The formula can be derived from a series expansion of $\kappa(q,
\beta)$ in \cite{graczyk:11:aihp}, which holds for any non-monotone
strictly stable process with index in a dense subset $\Cal A$ of
$(0,2)\setminus \mathbb Q$.  In the case of $X$, provided $\alpha\in
(1,2)\cap \Cal A$, the series can be reduced to the closed
form in \eqref{e:ladder}.  Then by continuity, \eqref{e:ladder} holds
for all $\alpha\in (1,2)$.  In the Appendix, we will give an
alternative proof of \eqref{e:ladder} without relying on the
continuity argument.

In the next section, as a preparation, some general results on first
passage of a \levy process are derived.  This section also collects
some standard results on stable processes.  In section \ref{s:proof},
Theorem~\ref{t:h} and its corollaries are proved.  In section \ref
{s:exact-sampling}, we show that $(\fpt_1, X_{\fpt_1-}, \jp {\fpt_1})$
can be sampled exactly.  It will be seen that the main issue is the
sampling of $h_1(x, \cdot)/v_1(x)$ for any fixed $x<1$, which is the
conditional \pdf of $\fpt_1$ given $X_{\fpt_1-} = x$.   The key is to
show that $h_1(x,t)$ can be decomposed as the sum of positive
functions $\phi_1(t)$, $\phi_2(t)$, \ldots.   Even though these
functions do not have a closed form, given $t>0$, each can be
evaluated in a finite number of steps, and for the exact sampling,
only a finite number of them have to be evaluated.  It is important to
keep in mind that these functions are constructed with the value of
$h_1(x,t)$ being intractable.  The decomposition then allows the
conditional \pdf of $\fpt_1$ to be sampled by the rejection
sampling method.

\section{Some general distributional results} \label{s:general}
We first consider \levy processes in general, and then specialize to
spectrally positive ones.

\subsection{Properties of first passage by a general \levy process}
\begin{prop} \label{p:fp}
  Let $X$ be a \levy process and $\levym(\dd x)$ its \levy measure.
  \begin{thmenum}
  \item \label{i:fp} \emph{(Distribution when $X$ jumps over a
      level).} 
    For every $c\ge 0$, $t>0$, $x\in\Reals$, $w\in\Reals$, and
    $y>c$,
    \begin{align} \nonumber
      &
        \pr{\fpt_c\in\dd t, X_{\fpt_c-}\in \dd x,
        X_{\fpt_c}\in \dd y, \rsup X_{\fpt_c-}\in\dd w
        }
      \\\label{e:first-passage-3}
      &=
        \cf{x\vee 0\le y\le c}\dd t\,\levym(\dd y-x)\,
        \pr{X_t\in \dd x, \rsup X_t\in\dd w}.
    \end{align}
  \item \label{i:fp2}
    For every $c\ge 0$, $\pr{X_{\fpt_c-} < X_{\fpt_c} = c} = 0$.
  \end{thmenum}
\end{prop}
\begin{remark}
  Part b) is known when $X$ is strictly stable with index $\alpha>1$
  (\cite{bertoin:96:cup}, Proposition VIII.8).
\end{remark}
\begin{proof}
  \ref{i:fp}
  The proof is standard so we only give a sketch of it (cf.\ \cite
  {bertoin:96:cup}, p.~76).  Given a Borel function  $f(t, x, y, w)\ge
  0$, $f(\fpt_c, X_{\fpt_c-}, X_{\fpt_c}, \rsup X_{\fpt_c-})
  \cf{X_{\fpt_c} > c} = \sum_{t: \jp_t\ne0} H_t(\jp_t)$, where $H_t(z)
  = f(t, X_{t-}, X_{t-}+z, \rsup X_{t-}) \cf{z> c - X_{t-}\ge 0, \rsup
    X_{t-}\le c}$.  Then by the compensation formula (\cite
  {bertoin:96:cup}, p.~7),
  \begin{align*} 
    \int f(t, x, y, w) \cf{y>c}
    \pr{\fpt_c\in\dd t, X_{\fpt_c-}\in\dd x, X_{\fpt_c} \in
      \dd y, \rsup X_{\fpt_c-}\in\dd w}
    =
    \int \mean [H_t(z)]\,\dd t\,\levym(\dd z).
  \end{align*}
  However, $\mean[H_t(z)] = \int f(t, x, x+z, w) \cf{z>c-x\ge 0, x\vee
    0\le w\le c} \pr{X_t\in\dd x, \rsup X_t\in\dd w}$.  Plug the
  equation into the right hand side (\rhs)~of the display.  Since $f$
  is arbitrary, by comparing he integrals on both sides,
  \eqref {e:first-passage-3} follows.

  \ref{i:fp2}  If 0 is not regular for $(0, \infty)$, then by the
  strong Markov property of $X$, there is a random $\rx>0$, such that
  $X_t\le X_{\fpt_c}$ for $t\in (\fpt_c, \fpt_c+\rx)$, implying
  $X_{\fpt_c} > c$.  Now suppose 0 is regular for $(0,\infty)$.  If
  $X_{\fpt_c}=c$, then $\fpt_c \ge \tau := \inf\{t: X_t=c,\,
  X_s<c\,\forall s<t\}$.  However, by the regularity of 0 and strong
  Markov property, $X_{t_n}>X_\tau=c$ for an infinite sequence $t_n
  \dto\tau$, implying $\fpt_c =\tau$.  Then  $\cf{X_{\fpt_c} = c >
    X_{\fpt_c-}} \le \sum_{t:\jp_t>0} \cf{X_t = c, X_s<c\,\forall
    s<t}$.  Then by following the argument for Proposition III.2(ii)
  in \cite {bertoin:96:cup} and noting that $X$ is not compound
  Poisson, the claim follows.
\end{proof}

In the next preliminary result, denote $\rsup \levym(x) =
\levym((x,\infty))$. 
\begin{prop} \label{p:fp-df}
  Suppose $\rsup\levym(0)>0$ and each $X_t$ has a \pdf.  Fix
  $c>0$ and define
  \begin{align} \label{e:first-passage-df}
    v_c(x) = \intzi h_c(x,t)\,\dd t,
  \end{align}
  where $h_c(x,t)$ is as in \eqref{e:hc}.  Let $D_c = \{\jp_{\fpt_c} >
  0\}$, i.e.\ the event that $X$ has a jump at the first passage at
  level $c$.
  \begin{thmenum}
  \item \label{i:v}
    $v_c(x)<\infty$ for a.e.\ $x\le c$ (in Lebesgue measure). 

  \item \label{i:ind}
    Conditionally on $D_c$, $X_{\fpt_c-}$ is concentrated on $\Omega_c
    = \{x\le c: \rsup\levym(c-x) v_c(x)>0\}$.  Moreover, conditionally
    on $D_c$ and $X_{\fpt_c-}=x\in\Omega_c$,  $(\fpt_c, \rsup
    X_{\fpt_c-})$ and $\jp_{\fpt_c}$ are independent, such that
    \begin{align*}
      \pr{\jp_{\fpt_c}\in\dd z\gv D_c, X_{\fpt_c-}=x}
      &=
        \frac{\cf{z>c-x}\levym(\dd z)}{\rsup\levym(c-x)},
      \\
      \pr{\fpt_c\in\dd t\gv D_c, X_{\fpt_c-}=x}
      &=
        \frac{h_c(x,t)}{v_c(x)},
    \end{align*}
    and for $w\in [x\vee 0, c]$
    \begin{gather*}
      \pr{\rsup X_{\fpt_c-}\le w\gv
        \fpt_c=t, D_c, X_{\fpt_c-}=x}
      =
      \frac{h_w(x,t)}{h_c(x,t)}.
    \end{gather*}
  \end{thmenum}
\end{prop}

\begin{proof}
  \ref{i:v} 
  Fix $-\infty < a < b\le c$.  By Fubini theorem
  \begin{align*}
    \int^b_a v_c(x)\,\dd x
    \le
    \int^b_a \,\dd x\intzi
    \frac{\pr{X_t\in \dd x}}{\dd x}\,\dd t
    =
    \intzi \pr{a\le X_t\le b}\,\dd t.
  \end{align*}
  By definition, if $X$ is transient, then the last integral is finite
  (\cite {bertoin:96:cup}. p.~32) and so $\int^b_a v_c<\infty$.  Since
  $a$ and $b$ are arbitrary, $v_c(x)<\infty$ for a.e.\ $x<c$.  If $X$
  is not transient, then it is recurrent, so $\rsup X_t\toi$ and
  $\rinf X_t \to-\infty$ \as (\cite {bertoin:96:cup}, p.~167--168). 
  Given $r>0$, let $\tau$ be an exponentially distributed random
  variable with mean $1/r$ and independent of $X$.  Then
  \begin{align*}
    \int^b_a \dd x\intzi e^{-r t} h_c(x,t)\,\dd t
    &=
    \intzi e^{-r t} \,\dd t\int^b_a\pr{X_t\in\dd x, \rsup X_t\le c}
    =
    r^{-1}\pr{\rsup X_\tau \le c, X_\tau\in [a,b]}
    \\
    &\stackrel{(*)}{=}
    r^{-1}
    \int \cf{0\le s\le c, y\ge0, a\le s-y\le b} \pr{\rsup X_\tau\in\dd
      s} \pr{-\rinf X_\tau \in \dd y}
    \\
    &\le
    r^{-1}
    \pr{\rsup X_\tau \in [0,c]} \pr{-\rinf X_\tau\in [(-b)\vee 0,
      c-a]},
  \end{align*}
  where $(*)$ is due to $\rsup X_\tau$ and $X_\tau-\rsup X_\tau\sim
  \rinf X_\tau$ being independent (\cite{bertoin:96:cup}, Theorem VI.5
  and Proposition VI.3).  As in the proof of Theorem VI.20 in \cite
  {bertoin:96:cup} or Theorem 3 in \cite{doney:06:aap}, let $r\dto
  0$.  By monotone convergence, $\int^b_a v_c\le \Cal U([0,c])
  \, \dual{\Cal U}([(-b)\vee 0, c-a])$, where $\Cal U$ (resp.\ $\dual
  {\Cal U}$) is the renewal measure of the ascending (resp.\
  descending) ladder height process of $X$.  Since both ladder
  processes are transient, the \rhs is finite, again yielding
  $v_c(x)<\infty$ for a.e.\ $x$.

  \ref{i:ind}
  By Proposition \ref{p:fp}\ref{i:fp2}, for $t>0$, $x\le c$, $x\vee
  0\le w\le c$, and $z>0$,
  \begin{align*} 
    &\pr{\fpt_c\in\dd t, X_{\fpt_c-}\in\dd x, \rsup X_{\fpt_c-}\le
      w, \jp_{\fpt_c}\in \dd z}
    =
    \cf{z>c-x}\dd t \,h_w(x,t) \,\dd x \,\levym(\dd z).
  \end{align*}
  Integrating over $t$ and $z$ yields $\pr{X_{\fpt_c-}\in\dd x, \rsup
    X_{\fpt_c-}\le w, D_c} = \rsup\levym(c-x) v_w(x)\, \dd x$.  In
  particular, letting $w=c$ gives $\pr{X_{\fpt_c-}\in\dd x,
    \jp_{\fpt_c}>0} = \rsup\levym(c-x)v_c(x)\, \dd x$.  This shows
  that conditionally on $D_c$, $X_{\fpt_c-}$ is concentrated on
  $\Omega_c$ and, together last display, also shows that for
  $x\in\Omega_c$,
  \begin{align*}
    &
    \pr{
      \fpt_c\in\dd t, \rsup X_{\fpt_c-}\le w, 
      \jp_{\fpt_c}\in\dd z\gv X_{\fpt_c-}\in\dd x, \jp_{\fpt_c}>0
    }
    \\
    &=
    \frac{h_w(x,t)\,\dd t}{h_c(x,t)}\times
    \frac{h_c(x,t)}{v_c(x)}\times
    \frac{\cf{z>c-x}\levym(\dd z)}{\rsup\levym(c-x)}.
  \end{align*}
  Then the rest of the claim easily follows.
\end{proof}

\subsection{The spectrally positive case}
Let $X$ be a spectrally positive \levy process that is not a
subordinator.  Then single points are not essentially polar for $X$,
whether the process has bounded variation (\cite{sato:99:cup}, Theorem
43.13) or not (\cite{bertoin:96:cup}, Corollary VII.5).  From
potential theory (\cite {bertoin:96:cup}, Section II.5), it follows
that $X$ has a bounded $q$-coexcessive version of resolvent density
$u^q(x)$ that satisfies
\begin{align} \label{e:u-fht}
  u^q(x)=\mean [e^{-q\fht_x}] u^q(0)
\end{align}
for $q>0$ and $x\in\Reals$, and if for every $t>0$, $X_t$ has a \pdf
$g_t$, then
\begin{align} \label{e:coexcessive}
  u^q(x) = \intzi e^{-q t} g_t(x)\,\dd t,
\end{align}
which can be extended to $q=0$ when $X$ is transient.  It will always
be assumed that $g_t$ is the unique version of \pdf that satisfies
\begin{align*}
  \int g_s(y) g_t(x-y)\,\dd y = g_{s+t}(x)
\end{align*}
for all $x$, $y\in\Reals$ and $s$, $t>0$.  Eq.~\eqref {e:coexcessive}
is stated in Remark 41.20 of \cite{sato:99:cup} under the assumption
that $g_t$ is bounded and continuous.  It is probably known that
\eqref {e:coexcessive} holds in general.  However, we could not find
an explicit proof in literature, so for convenience, one is given in
the Appendix.

To evaluate $h_c(x,t)$ defined in \eqref{e:hc} for stable processes,
the following result will be used.
\begin{prop} \label{p:pre-passage}
  Suppose that each $X_t$ has a \pdf $g_t$.  Then given
  $x<c$,
  \begin{align} \label{e:pre-passage}
    h_c(x,t)
    &= g_t(x) - \int^t_0 f_c(s) g_{t-s}(x-c)\,\dd s
    \\\label{e:pre-passage2}
    &= g_t(x) - \int^t_0 f_{x-c}(s) g_{t-s}(c)\,\dd s.
  \end{align}
  Furthermore, if $x>0$, then $h_c(-x,\cdot)$ is the convolution of
  $h_c(0, \cdot)$ and $f_{-x}$, i.e.,
  \begin{align} \label{e:first-passage-conv}
    h_c(-x, \cdot) = h_c(0, \cdot) * f_{-x}.
  \end{align}
\end{prop}

\begin{proof}
  Since $X$ has no negative jumps and $\fht_c > \fpt_c$ \as (\cite
  {bertoin:96:cup}, Proposition VIII.8(ii)), for each $A\subset
  (-\infty,c)$, $\cf{X_t\in A, \rsup X_t>c} =\cf{X_t\in A, \fht_c<t}$ 
  \as.  Then by the strong Markov property of $X$, for any bounded
  continuous function $k(x)\ge 0$ with support in $(-\infty, c)$,
  \begin{multline*}
    \mean[k(X_t) \cf{\rsup X_t>c}]
    =
    \mean[k(X_t) \cf{\fht_c<t}]
    =\int^t_0 \mean[k(X_{t-s}+c)]\pr{\fht_c\in\dd s}
    \\
    =\int^t_0 
    \Sbr{\int k(x+c) g_{t-s}(x)\,\dd x} f_c(s)\,\dd s
    =
    \int k(x) \Sbr{\int^t_0 f_c(s)g_{t-s}(x-c)\,\dd s}\,\dd x.
  \end{multline*}
  On the other hand,
  \begin{align*}
    \int k(x) h_c(x,t)\,\dd x
    =
    \mean[k(X_t) \cf{\rsup X_t\le c}]
    =
    \mean[k(X_t)]
    -
    \mean[k(X_t)\cf{\rsup X_t>c}].
  \end{align*}
  The two displays together with $\mean[k(X_t)] = \int k(x) g_t(x)
  \,\dd x$ give \eqref{e:pre-passage}.  Eq.~\eqref {e:pre-passage2} is
  essentially shown on p.~4/10 of \cite {michna:15:ecp} (also see
  \cite{peskir:08:ecp}).  Given $x>0$, write $m_{-x}(t) = g_t(-x)$ as
  a function of $t$ and $\LT{m_{-x}}$ its Laplace transform.  By
  \eqref {e:u-fht} and \eqref {e:coexcessive}, $\LT{m_{-x}} = u^q(-x)
  = \LT{f_{-x}} u^q(0) = \LT{f_{-x}} \LT{m_0}$.  Then $m_{-x} =
  m_0*f_{-x}$.  Also, $f_{-x-c} = f_{-c}*f_{-x}$.  Plugging the two
  equations into \eqref{e:pre-passage2} yields \eqref
  {e:first-passage-conv}.
\end{proof}

\subsection{Preliminaries on stable processes}
From now on let $X$ be a spectrally positive and strictly stable
process with index $\alpha\in(1,2)$ satisfying \eqref
{e:stable-normal}.  Then the \levy measure of $X$ is
\begin{align} \label{e:Levy-stable}
  \levym(\dd x) = \frac{\cf{x>0} x^{-\alpha-1}}{\Gamma(-\alpha)}
  \,\dd x
\end{align}
(\cite{feller:71:jws}, p.~570).  By scaling and \cite
{sato:99:cup}, p.~88, $g_t$ has power series expansion on $\Reals$,
\begin{align} \label{e:pdf-g}
  g_t(x) 
  = t^{-\ra} g_1(t^{-\ra} x)
  =
  \nth{\alpha\pi}
  \sumoi k \frac{\Gamma(k\da)}{(k-1)!} \sin(k\pi\da)
  t^{-k\da} x^{k-1}.
\end{align}
By \cite{bertoin:96:cup}, Theorem VII.1, $(\fht_{-x})_{x\ge 0}$ is
strictly stable with index $\ra$, such that 
\begin{align} \label{e:T_x}
  \mean (e^{-q \fht_{-x}})
  =
  \exp(-xq^\ra), \quad q \ge 0.
\end{align}
By scaling and \cite{sato:99:cup}, p.~88, or by Kendall's identity,
for $x>0$ and $t>0$,
\begin{align} \label{e:pdf-f}
  f_{-x}(t) = x^{-\alpha} f_{-1}(x^{-\alpha} t)
  =
  \nth\pi \sumoi k (-1)^{k-1}
  \frac{\Gamma(k\da+1)}{k!} 
  \sin(\pi k\da) x^k t^{-k\da-1}.
\end{align}

From \eqref{e:pdf-g}, $g_t(x)$ as a function of $(x,t)$ can be extended
from $\Reals\times (0,\infty)$ to $\Coms\times (\Coms\setminus
(-\infty, 0])$, such that for each fixed $x\in \Coms$, the extension
is an analytic function of $t\in \Coms\setminus(-\infty,0]$, and for
each fixed $t\in \Coms\setminus (-\infty,0]$, it is an analytic
function of $x\in\Coms$.  By \eqref{e:pdf-f},  $f_{-x}(t)$ can be
similarly extended from $(0,\infty)\times (0,\infty)$ to $\Coms\times 
(\Coms\setminus (-\infty, 0])$.  However, the extension is not the
same as $f_{-x}(t)$ for $(x,t)\in (-\infty, 0)\times (0,\infty)$.
Indeed, for $x<0$ and $t>0$, the extension necessarily has the power
series expansion \eqref{e:pdf-f}.  On the other hand, for $x<0$, the
power series of $f_{-x}(t)$ is quite different (\cite{simon:11:sto},
Proposition 3).

Finally, for $s\in \Reals$ (\cite{uchaikin:99:vsp}, Section 5.6),
\begin{align} \label{e:moment-g}
  \intzi x^s g_1(x)\,\dd x = 
  \begin{cases}
    \displaystyle
    \frac{\Gamma(s) \Gamma(1-s\da)}{\Gamma(s(1-\ra))
      \Gamma(1-s(1-\ra))} & \text{if } s\in (-1, \alpha)
    \\
    \infty & \text{else}
  \end{cases}
\end{align}
and 
\begin{align} \label{e:moment-f}
  \intzi t^s f_{-1}(t)\,\dd t=
  \begin{cases}
    \displaystyle
    \frac{\Gamma(1-s\alpha)}{\Gamma(1-s)} & \text{if } s<\ra
    \\
    \infty & \text{else.}
  \end{cases}
\end{align}

\section{Proof of main results} \label{s:proof}
\subsection{Initial deduction by Laplace transform} \label{ss:laplace}
We need the following formulas from \cite{simon:11:sto}.  Given $x>0$,
\begin{align} \label{e:Simon}
  \mean(e^{-q \fht_x}) 
  =
  \mean(e^{-q x^\alpha \fht_1})
  =
  F_1(q^\ra x) - \alpha F'_\alpha(q^\ra x),
\end{align}
where $F_a(x) = E_a(x^a) := E_{a,1}(x^a)$ and for fixed $a>0$ and
$d\in\Coms$, the following function
of $z\in\Coms$
\begin{align*}
  E_{a,d}(z) = \sumzi n \frac{z^n}{\Gamma(a n +d)}
\end{align*}
is known as the Mittag-Leffler function.  Then $F_1(q^\ra x) =
E_1(q^\ra x) = e^{q^\ra x}$ and from
\begin{align*}
  F'_\alpha(z)
  =
  \Grp{\sumzi n \frac{z^{\alpha n}}{\Gamma(1+ \alpha n)}}'
  =
  \sumoi n \frac{z^{\alpha n-1}}{\Gamma(\alpha n)},
\end{align*}
it follows that
\begin{align*}
  \alpha F'_\alpha(q^\ra x)
  =
  \alpha \sumoi n
  \frac{q^{n-\ra} x^{\alpha n - 1}}{\Gamma(\alpha n)}.
\end{align*}

Fix $x<1$.  We seek the Laplace transform of $h_1(x,\cdot)$.  For
brevity, put 
\begin{align*}
  \hbar(t)= h_1(x, t).
\end{align*}
\begin{prop} \label{p:LT}
  For $q>0$,
  \begin{align}\label{e:LP-u}
    \LT\hbar(q)
    &=
    e^{-q^\ra(1-x)} \sumoi n \frac{q^{n-1}}{\Gamma(\alpha n)}
    -
    \sumoi n\frac{q^{n-1} (x\vee 0)^{\alpha n-1}}{\Gamma(\alpha n)}.
  \end{align}
\end{prop}

\begin{remark}
  By \eqref{e:first-passage-df}, $v_1(x) = \LT\hbar(0+)$, which
  together with \eqref{e:LP-u} yields  \eqref{e:v-stable}.
\end{remark}
\begin{proof}[Proof of Proposition \ref{p:LT}]
  By \eqref{e:pre-passage} in Proposition \ref{p:pre-passage}, the
  Laplace transform of $\hbar$ is
  \begin{align*}
    \LT\hbar(q)
    =\intzi e^{-q t} \Sbr{
    g_t(x) - \int^t_0 f_1(s) g_{t-s}(x-1)\,\dd s
    }\,\dd t
    &=
    u^q(x) - \mean(e^{-q \fht_1}) u^q(x-1).
  \end{align*}
  By \eqref{e:coexcessive}, scaling, and \eqref{e:pdf-g},
  \begin{align*}
    u^q(0) = \intzi e^{-qt} g_t(0)\,\dd t
    = g_1(0) \intzi e^{-q t} t^{-\ra}\,\dd t
    =  \alpha^{-1} q^{\ra-1}.
  \end{align*}
  Then by \eqref{e:u-fht},
  \begin{align}  \label{e:lt-h}
    \LT\hbar(q)=
    \alpha^{-1} q^{\ra-1}
    [\mean(e^{-q \fht_x}) - \mean(e^{-q \fht_{x-1}}) \mean(e^{-q
    \fht_1})].
  \end{align}
  Since $x-1<0$, by \eqref{e:T_x}, $\mean(e^{-q \fht_{x-1}}) =
  e^{(x-1) q^\ra}$.  If $x\le 0$, then $\mean(e^{-q
    \fht_x}) = e^{x q^\ra}$ as well, and so applying 
  \eqref{e:Simon} to $\mean(e^{-q\fht_1})$, 
  \begin{align*}
    \mean(e^{-q \fht_x}) - \mean(e^{-q \fht_{x-1}}) \mean(e^{-q\fht_1})
    &= 
    e^{q^\ra x} 
    -
    e^{q^\ra(x-1)}\Grp{
      e^{q^\ra} - \alpha\sumoi
      n\frac{q^{n-\ra}}{\Gamma(\alpha n)}
    }
    \\
    &=
    \alpha e^{q^\ra(x-1)} \sumoi n \frac{q^{n-\ra}}
    {\Gamma(\alpha n)}.
  \end{align*}
  On the other hand, if $x>0$, then applying \eqref{e:Simon} to both
  $\mean(e^{-q\fht_x})$ and $\mean(e^{-q\fht_1})$,
  \begin{align*}
    &
    \mean(e^{-q \fht_x}) - \mean(e^{-q \fht_{x-1}}) \mean(e^{-q \fht_1})
    \\
    &= 
    e^{q^\ra x} - \alpha \sumoi n \frac{q^{n-\ra} x^{\alpha
      n - 1}}{\Gamma(\alpha n)}
    -
    e^{q^\ra(x-1)}\Grp{
      e^{q^\ra} - \alpha \sumoi n \frac{q^{n-\ra}}
      {\Gamma(\alpha n)}
    }
    \\
    &=\alpha e^{q^\ra(x-1)} \sumoi n
    \frac{q^{n-\ra}}{\Gamma(\alpha n)}
    -
    \alpha \sumoi n \frac{q^{n-\ra} x^{\alpha
    n - 1}}{\Gamma(\alpha n)}.
  \end{align*}
  The above two identities for $\mean(e^{-q \fht_x}) - \mean(e^{-q
    \fht_{x-1}}) \mean(e^{-q \fht_1})$ combined with \eqref{e:lt-h}
  then lead to \eqref{e:LP-u}.
\end{proof}

Given $x<1$, $f_{x-1}$ belongs to $C^\infty_0([0,\infty))$, the family
of infinitely differentiable functions on $[0,\infty)$ with derivative
of any order equal to zero at $0$ and $\infty$.  Since $e^{-q^\ra
  (1-x)}$ is the Laplace transform of $f_{x-1}$, in view of
\eqref{e:LP-u} and the relationship between Laplace transform and
differentiation, if $x<0$, then it is possible to show \eqref
{e:h-series} by interchanging Laplace transform and the infinite
summation on the \rhs of \eqref {e:LP-u}.  However, as noted in the
introduction, for $x>0$ the approach fails to work.  In our proof,
\eqref{e:h-series} is first established for $x<a$ and $t>0$, where
$a<0$ is some constant.  The argument based on interchanging Laplace
transform and infinite summation is carried out.  Then the general
case is resolved by analytic extension.

\subsection{Proof of theorem}
Fixing $t>0$, regard $h_1(1-x,t)$ as a function of $x$.  We need a
preliminary estimate of the domain it can be analytically extended to.
Recall that a domain is a connected open set in $\Coms$. 
\begin{lemma} \label{l:analytic}
  Given $t>0$, the mapping $x\to h_1(1-x,t)$ can be analytically
  extended from $(0,\infty)$ to $\Omega :=\{z\in \Coms: |\arg z| <
  \pi/2 - \pi/(2\alpha)\}$.
\end{lemma}
\begin{proof}
  The following fact will be used.  Let $D\subset \Coms$ be a domain
  and $J\subset\Reals$.  Suppose $m(z,\lambda)$ is a measurable
  function on $D\times J$ and $\nu$ is a measure on $J$.  If $m(\cdot,
  \lambda)$ is analytic in $D$ for each $\lambda\in J$, and the
  mapping $z\to \int |m(z,\lambda)|\nu(\dd \lambda)$ is bounded on any
  compact subset of $D$, then by Fubini's theorem and Morera's
  theorem (\cite{rudin:87:mcgraw}, p.~208), $M(z)=\int m(z,
  \lambda)\nu(\dd \lambda)$ is analytic on $D$.

  Given $t>0$, by Proposition \ref{p:pre-passage},  $h_1(1-x, t) =
  g_t(1-x) - \int^t_0 f_1(t-s) g_{s}(-x) \,\dd s$.  Since $g_t$ can
  be analytically extended to $\Coms$, it suffices to show that $x\to
  \int^t_0 g_s(-x) f_1(t-s)\,\dd s$ can be analytically extended to
  $\Omega$.  By Kendall's identity
  \begin{align*}
    \int^t_0 g_s(-x) f_1(t-s)\,\dd s
    &=
      \nth x \int^t_0 s f_{-x}(s) f_1(t-s)\,\dd s.
  \end{align*}
  The Fourier transform of $f_{-x}$ is $\FT f_{-x}(\lambda) =
  \LT{f_{-x}}(-\iunit\lambda) = e^{-(-\iunit \lambda)^\ra x}$,
  $\lambda\in\Reals$, where $-\pi < \arg(-\iunit\lambda)\le
  \pi$.  Then $|\FT f_{-x}(\lambda)| = e^{-\Re(-\iunit
    \lambda)^\ra x} = e^{-|\lambda|^\ra x\cos a}$ with $a =
  \pi/(2\alpha)$.  As $\cos a >0$, Fourier inversion can be applied to
  get
  \begin{align*}
    f_{-x}(s)
    &=
    \nth{2\pi} \intii e^{-\iunit \lambda s  -
      (-\iunit\lambda)^\ra x}\, \dd \lambda
  \end{align*}
  and Fubini theorem can be applied to get
  \begin{align*}
    \int^t_0 g_s(-x) f_1(t-s)\,\dd s
    &=
    \nth {2\pi x} \intii \psi(\lambda) e^{-(-\iunit\lambda)^\ra
      x}\,\dd  \lambda,
  \end{align*}
  where $\psi(\lambda) = \int^t_0 s e^{-\iunit\lambda s} f_1(t-s)\,\dd
  s$ is bounded.  Given a compact $C\subset \Omega$, $\theta_0 :=
  \sup_{z\in C} |\arg(z)| < \pi/2 - a$ and $r_0 := \inf_{z\in
    C}|z|>0$.   For $z = r e^{\iunit\theta}\in C$, 
  $\Re((-\iunit \lambda)^\ra z) = \lambda^\ra r \cos(\theta \pm a)$,
  where the sign of $a$ is opposite to that of $\lambda$.  By $|\theta
  \pm a| \le \theta_0 + a < \pi/2$, $\Re((-\iunit \lambda)^\ra z) \ge
  c\lambda^\ra $ with $c = r_0\cos(\theta_0 + a)>0$.  Then $\intii
  \psi(\lambda) e^{-(-\iunit\lambda)^\ra z}\,\dd \lambda$ is bounded
  on $C$.  As remarked at the beginning, this yields the proof.
\end{proof}

\begin{lemma} \label{l:continuity}
  For every $x<1$ and $t>0$, the series in \eqref {e:h-power} and
  \eqref{e:h-series} converge absolutely and are equal to each other,
  and as functions of $(x,t)$ can be extended to $\Coms\times (\Coms
  \setminus (-\infty,0])$, such that for each fixed $x\in \Coms$, the
  extended function is analytic in $t\in \Coms\setminus(-\infty,0]$,
  and for each fixed $t\in \Coms\setminus (-\infty,0]$, the extended
  function is analytic in $x\in\Coms$.
\end{lemma}
\begin{proof}
  By \eqref{e:pdf-f}, the series in \eqref{e:h-series} is
  \begin{align*}
    &
    \sumoi n \frac{1}{\Gamma(\alpha n)}
    \frac{\dd^{n-1}}{\dd t^{n-1}}
    \Grp{
      \nth\pi
      \sumoi k
      (-1)^{k-1}
      \frac{\Gamma(k\da+1)}{k!} 
      \sin(\pi k\da) (1-x)^k t^{-k\da-1}
    }\\
    &=
    \nth\pi
    \sumoi n \frac{(-1)^{n-1}}{\Gamma(\alpha n)}
    \sumoi k
    (-1)^{k-1}
    \frac{\Gamma(k\da + n)}{k!} 
    \sin(\pi k\da) (1-x)^k t^{-k\da-n}.
  \end{align*}
  Then to show the entire lemma, it suffices to show that series on
  \eqref{e:h-power} converges absolutely.  Letting $M = [1-(x\wedge
  0)] t^{-1}$, the sum of the absolute values of the terms in the
  series is less than
  \begin{align}\nonumber
    &\hspace{-1cm}
    \sumoi{k,n}
    \frac{\Gamma(k\da + n)}{\Gamma(\alpha n) k!} M^{k\da + n}
    =
    \intzi \sumoi{k,n}
    \frac{s^{k\da + n-1}} {\Gamma(\alpha n)
      k!} e^{-s/M}\,\dd s
    \\\label{e:h-series-bound}
    &=
    \intzi \Grp{\sumoi n \frac{s^{n-1}}{\Gamma(\alpha n)}} 
    \Grp{\sumoi k\frac{s^{k\da}}{k!}} e^{-s/M}\,\dd s
    \le
    \intzi
    E_{\alpha,\alpha}(s) e^{s^\ra-s/M}\,\dd s.
  \end{align}
  From (22) on p.~210 of \cite{erdelyi:55:mcgraw}, as $s\toi$,
  $E_{\alpha, \alpha}(s) e^{s^\ra - s/M} = O(e^{2s^\ra-s/M})$.
  Therefore the last integral is finite, yielding the desired absolute
  convergence.
\end{proof}

Given $x>0$, denote
\begin{align*}
  h_x(t) = h_1(1-x,t)
\end{align*}
and regard it as a function of $t>0$.  A key ingredient of the proof
of Theorem \ref{t:h} is to show that $h_x(t)$ is identical to 
\begin{align} \label{e:u}
  \omega_x(t)
  =
  \sumzi n \frac{f\Sp n_{-x}(t)}
  {\Gamma(\alpha n + \alpha)}.
\end{align}
For this purpose the following lemmas are needed.
\begin{lemma} \label{l:basic}
  Given $x>0$, $h_x(t)$ is a bounded and continuous function of $t>0$
\end{lemma}
\begin{proof}
  By Proposition \ref{p:pre-passage}, $h_x(t) \le m_x(t):=g_t(x)$.
  From \eqref{e:pdf-g}, $m_x(t)$ is continuous in $t>0$ and is bounded
  on $[a,\infty)$ for any $a>0$.  On the other hand, by scaling and
  $g_1(z) = O(z^{-1-\alpha})$ as $z\toi$, $m_x(t) = g_t(x) = t^{-\ra}
  g_1(t^{-\ra} x) = O(t)$ as $t\to0$.  Therefore $m_x(t)$ is bounded
  on $(0,\infty)$ and so is $h_x(t)$.  Next, observe that both
  $f_{-x}(t)$ and $m_1(t)$ can be extended into uniformly continuous
  and integrable functions on the entire $\Reals$ with values on
  $(-\infty, 0]$ equal to 0.  As a result, $f_{-x}*m_1$ is continuous.
  Then by Proposition \ref{p:pre-passage}, $h_x(t) = m_x(t) -
  (f_{-x}*m_1)(t)$ is continuous.  
\end{proof}

\begin{lemma} \label{l:omega-x}
  Let $x_0 = 1/\cos(\pi/(2\alpha))$ and fix $x>x_0$.
  \begin{thmenum}
  \item \label{i:omega}
    The following function is bounded in $t>0$,
    \begin{align*}
      \varsigma_x(t) =
      \sumzi n \frac{|f\Sp n_{-x}(t)|}{\Gamma(\alpha n + \alpha)}.
    \end{align*}
  \item \label{i:omega1}
    $\omega_x(t)$ is a bounded and continuous function of $t>0$.
  \item \label{i:omega-lt}
    $\LT{\omega_x}(q) = \LT{h_x}(q)$.
  \end{thmenum}
\end{lemma}
\begin{proof}
  \ref{i:omega}
  From the bound on $|\FT f_{-1}(\lambda)|$ in the proof of Lemma
  \ref{l:analytic},  it follows that for any $n\ge 0$, $\intii
  |\lambda|^n |\FT f_{-1}(\lambda)|\,\dd \lambda <\infty$, so by
  Fourier inversion
  \begin{align*}
    f\Sp n_{-1}(t)
    =
    \frac{(-\iunit)^n}{2\pi} \intii \lambda^n e^{-\iunit \lambda t}
    \FT f_{-1}(\lambda)\,\dd \lambda
  \end{align*}
  and hence
  \begin{align*}
    \sup_t |f\Sp n_{-1}(t)|
    &\le
    \nth{2\pi}
    \intii |\lambda|^n  |\FT f_{-1}(\lambda)|\,\dd \lambda
    =
    \nth\pi
    \intzi \lambda^n e^{-\lambda^\ra/x_0}\,\dd\lambda
    \\
    &=
    \frac{\alpha}{\pi}
    \intzi \lambda^{\alpha n + \alpha-1} e^{-\lambda/x_0}\,\dd \lambda
    =
    \frac{\alpha\Gamma(\alpha n + \alpha)  x^{\alpha n + \alpha}_0}{\pi}.
  \end{align*}
  Since $f_{-x}(t) = x^{-\alpha} f_{-1}(x^{-\alpha} t)$, then $f\Sp
  n_{-x}(t) = x^{-\alpha n -\alpha} f\Sp n_{-1}(x^{-\alpha} t)$.  As a
  result,
  \begin{align*}
    \varsigma_x(t)
    =
    \sumzi n \frac{x^{-\alpha n -\alpha} |f\Sp n_{-1}(x^{-\alpha}
    t)|}{\Gamma(\alpha n + \alpha)}
    \le
    \frac\alpha\pi\sumzi n (x_0/x)^{\alpha n + \alpha},
  \end{align*}
  and hence for $x>x_0$, $\varsigma_x(t)$ is bounded in $t>0$.

  \ref{i:omega1}
  From \ref{i:omega}, it follows that $\omega_x(t)$ is bounded.  The
  continuity of $\omega_x(t)$ is implied in Lemma \ref{l:continuity}.
  
  \ref{i:omega-lt} By \ref{i:omega}, for $x>x_0$, the summation and
  integration can interchange in the calculation of $\LT{\omega_x}(q)$
  to yield
  \begin{align*}
    \LT{\omega_x}(q)
    =
    \sumzi n \frac{\LT{f\Sp n_{-x}}(q)}{\Gamma(\alpha n + \alpha)}
    =
    \sumzi n \frac{q^n\LT{f_{-x}}(q)}{\Gamma(\alpha n + \alpha)}
    =
    e^{-x q^\ra} \sumzi n \frac{q^n}{\Gamma(\alpha n + \alpha)}.
  \end{align*}
  Since $x>1$, from Proposition \ref{p:LT}, the \rhs is the Laplace
  transform of $h_x(\cdot) = h_1(1-x,\cdot)$.
\end{proof}

\begin{proof}[Proof of Theorem \ref{t:h}]
  Again write $h_x(t) = h_1(1-x, t)$.  By Lemma \ref{l:omega-x}, for
  $x>x_0>1$, $\omega_x(t)$ is bounded and $\LT{\omega_x}(q) =
  \LT{h_x}(q)$ for all $q>0$.  Then by the one-to-one correspondence 
  between bounded continuous functions and their Laplace transforms
  $h_x(t) = \omega_x(t)$.  Thus
  \begin{align*}
    h_1(1-x,t)
    &=
    \sumzi n \frac{f\Sp n_{-x}(t)}{\Gamma(\alpha n + \alpha)}
    \\
    &=
    \nth\pi
    \sumoi{k,n} \frac{\Gamma(k\da + n)}
    {\Gamma(\alpha n) k!}
    (-1)^{k+n} x^k \sin(\pi k\da) t^{-k\da - n}.
  \end{align*}
  Fix $t>0$ and treat $x$ as the only variable.  By Lemma \ref
  {l:analytic}, $h_1(1-x,t)$ can be analytically extended from $(0,
  \infty)$ to a domain $\Omega\subset \Coms$ containing $(0, \infty)$,
  while by Lemma \ref {l:continuity}, the two series in the display
  converge absolutely and can be analytically extended to the entire
  $\Coms$.  Since $h_1(1-x,t)$ and the two series agree on $(x_0,
  \infty)$, they must be equal on $\Omega$, in particular, on the
  entire $(0, \infty)$.  It follows that for every fixed $t>0$, \eqref
  {e:h-power} and \eqref{e:h-series} hold for all $x<1$.  This
  completes the proof of \eqref{e:h-power} and \eqref{e:h-series}.
  The rest of the theorem follows by combining \eqref {e:h-power} and
  \eqref{e:h-series} with Proposition \ref{p:fp} and
  \eqref{e:Levy-stable}.
\end{proof}

\subsection{Proofs of corollaries}
\begin{proof}[Proof of Corollary \ref{c:h}]
  From $\Gamma(z) = \intzi s^{z-1} e^{-s}\,\dd s$ and the absolute
  convergence of \eqref{e:h-series},
  \begin{align*}
    h_1(x,t)
    &=
    \nth\pi \intzi e^{-s}
    \sumoi{k,n}\frac{s^{k\da + n-1}}{\Gamma(\alpha n) k!}
    (-1)^{k+n} (1-x)^k \sin(\pi k\da) t^{-k\da -
      n}\,\dd s
    \\
    &=
    \nth\pi\intzi e^{-s}
    \Sbr{\sumoi k \frac{((x-1)(s/t)^\ra)^k \sin(\pi
        k\da)}{k!}
    }
    \Sbr{\sumoi n \frac{s^{n-1} (- 1/t)^n}{\Gamma(\alpha
        n)} 
    }\,\dd s
    \\
    &=
    \nth\pi \intzi 
    \Grp{\Im e^{-s - (1-x)(-s/t)^\ra}}
    (-1/t) E_{\alpha, \alpha}(-s/t)\,\dd s.
  \end{align*}
  By change of variable the integral representation follows.
\end{proof}

To prove the other corollaries, $h_c(x,t)$ is treated as a function of
$c$ as well as $x$ and $t$.  From \eqref {e:h-power} and the scaling
relationship \eqref{e:h-scaling}, 
\begin{align} \label{e:h-series-c}
  h_c(x,t)
  &=
  \nth c\sumzi n \frac{c^{\alpha n + \alpha}
    f\Sp n_{x-c}(t)}{\Gamma(\alpha n + \alpha)}
  \\\label{e:h-power-c}
  &=
  \nth\pi
  \sumoi{k,n} \frac{\Gamma(k\da + n)}{\Gamma(\alpha n) k!}
  (-1)^{k+n} (c-x)^k c^{\alpha n-1} \sin(\pi k\da) t^{-k\da -
    n}.
\end{align}
Both series converge absolutely for given $t>0$.  

\begin{proof}[Proof of Corollary \ref{c:nonzero}]
  Fix $x<1$ and $t>0$.  Put $\varphi(c)= h_c(x,t)>0$ and $c_0 = x\vee
  0$.  We shall show that $\varphi(c)>0$ for all $c>c_0$.  From its
  definition in \eqref{e:hc}, $\varphi(c)$ is increasing on $(c_0,
  \infty)$.  By Proposition \ref{p:pre-passage}, $\varphi(c) = g_t(x)
  -  \int^t_0 k_c(s)\,\dd s$, where $k_c(s) = f_{x-c}(s) g_{t-s}(c)$.
  For $c>c_0 + 1$, $f_{x-c}(s) = (c-x)^{-\alpha} f_{-1}((c -
  x)^{-\alpha} s)\le \sup f_{-1}$ and as $s\uto t$, $g_{t-s}(c) =
  (t-s)^{-\ra} g_1(c (t-s)^{-\ra}) = (t-s)^{-\ra} O((t-s)^{1+\ra}) =
  O(t)$.  As $c\toi$, $k_c(s)\to0$ for $s\in (0,t)$.  Then by
  dominated convergence, $\varphi(c)\to g_t(x)>0$.  On the other hand,
  from \eqref{e:h-power-c}, $\varphi(c)$ can be analytically extended
  to $\Coms\setminus (-\infty, 0]$.  If $\varphi(c)=0$ for some
  $c>c_0$, then by monotonicity, $\varphi(z) = 0$ for all $z\in (c_0,
  c)$.  Then by analyticity, $\varphi(z)=0$ for all $z>c_0$, yielding
  $\varphi(z)\to0$ as $z\toi$, a contradiction.
\end{proof}

\begin{proof}[Proof of Corollary \ref{c:joint-max}]
  Fix $t=1$ and $x$.  Then each term in the series \eqref{e:h-power-c}
  is a function of $c$, denoted $d_{k,n}(c)$.  It is seen that
  $d'_{k,n}(c)$ is the $(k,n)^\mathrm{th}$ term in the series \eqref
  {e:h-diff}.  For each bounded interval $I=[a,b]\subset (x\vee 0,
  \infty)$, letting $M = b+|x|$, for all $c\in I$ and $k$, $n\ge 1$,
  \begin{align*}
    |d'_{k,n}(c)|\le D_{k,n}:=
    \frac{\Gamma(k\da + n)}{\Gamma(\alpha n) k!}
    (k + \alpha n -1) M^{k+\alpha n -1}/a.
  \end{align*}
  By argument similar to that for Lemma \ref{l:continuity}, $\sumoi
  {k,n} D_{k,n} < \infty$.  As a result, $\partial h_c(x,1)/\partial c
  = \pi^{-1}\sumoi{k,n} d'_{k,n}(c)$ for $c\in I$.  Since $I$ is
  arbitrary, then \eqref{e:h-diff} holds for all $c>x\vee0$, as
  claimed.
\end{proof}

\begin{proof}[Proof of Corollary \ref{c:pre-max}]
  This is immediate from Proposition \ref{p:fp-df} and the fact that
  being spectrally positive with infinite variation $X$ does not
  creep, i.e., $\jp_{\fpt_c}>0$ \as (\cite {doney:07:sg-b},
  p.~64).
\end{proof}

\begin{proof}[Proof of Corollary \ref{c:running-max}]
  Let
  \begin{align*}
    a(x, c, r)
    =
    \frac{\pr{X_1\in\dd x, \rsup X_1\in\dd c, \rsup G_1\le r}}
    {\dd x\,\dd c}, \quad
    b(x, c, t)
    =
    \frac{\partial h_c(x,t)}{\partial c}.
  \end{align*}
  Although Corollary \ref{c:joint-max} provides a series expression of
  $b(x,c,t)$, it is not very useful here.  Instead, by
  \eqref{e:pre-passage2}, for $x<c$,
  \begin{align} \nonumber
    b(x,c,t)
    &=
    -\frac{\partial}{\partial c}
    \Sbr{\int^t_0 f_{x-c}(s) g_{t-s }(c)\,\dd s}
    \\\label{e:pre-passage3}
    &=
    -\int^t_0
    \Sbr{\frac{\partial f_{x-c}(s)}{\partial c} g_{t-s}(c)
      +f_{x-c}(s) g'_{t-s}(c)
    }
    \,\dd s,
  \end{align}
  where the interchange of integration and differentiation on the second
  line is justified by the uniform boundedness of $\partial(f_{x-c}(s)
  g_{t-s}(c))/\partial c$ as a function of $(c,s)$ on any compact
  set in $(x\vee 0,\infty)\times [0,\infty)$.
  
  Given $r\in (0,1)$, for each $\rx\in (0,1-r)$, by conditioning on
  $X_r$ and $X_{r+\rx}$ and the Markov property of $X$,
  \begin{align*}
    &
    \pr{X_1\in\dd x, \rsup X_1\in\dd c, r<\rsup G_1\le r+\rx}
    \\
    &=
    \int_{u<c, v<c-u} \pr{X_r\in\dd u, \rsup X_r<c}\,
      \pr{X_\rx\in \dd v, \rsup X_\rx \in \dd c - u}
    \\
    &\hspace{2.2cm}\times
    \pr{X_{1-r-\rx}\in \dd x - u-v, \rsup X_{1-r-\rx}<c-u-v}
    \\
    &=
    \int_{u<c, v<c-u} \pr{X_r\in\dd u, \rsup X_r\le c}\,
      \pr{X_\rx\in \dd v, \rsup X_\rx \in \dd c - u}
    \\
    &\hspace{2.2cm}\times
    \pr{X_{1-r-\rx}\in \dd x - u-v, \rsup X_{1-r-\rx}\le c-u-v},
  \end{align*}
  where the second equation is due to $\rsup X_r$ and $\rsup X_{1-r
    -\rx}$ having continuous distributions according to Corollary
  \ref{c:joint-max}.  Make change of variables $y = c - u$ and
  $z=c-u-v$.  Then divide both sides of the above display by $\rx\,\dd
  x\,\dd c$ and use \eqref{e:hc} and Corollary \ref{c:joint-max} to get
  \begin{align} \label{e:a-diff}
    \frac{a(x, c, r+\rx) - a(x,c,r)}{\rx}
    =
    \int_{y>0, z>0} h_c(c-y, r)
    \frac{b(y-z, y, \rx)}{\rx} h_z(x-c+z, t - r -\rx)\,\dd y\,\dd z.
  \end{align}
  From \eqref{e:pre-passage3}, it follows that for $y>0$ and $z>0$,
  \begin{align*}
    b(y-z, y, \rx)
    =
    -\int^\rx_0 \Sbr{
      g_{\rx-s}(y)\frac{\partial f_{-z}(s)}{\partial
      z}+g'_{\rx-s}(y)f_{-z}(s) 
    }\,\dd s.
  \end{align*}
  By $f_{-z}(s) = z^{-\alpha}f_{-1}(z^{-\alpha} s)$,
  \begin{align*}
    \frac{\partial f_{-z}(s)}{\partial z}
    =
    -\alpha z^{-\alpha - 1}[f_{-1}(z^{-\alpha} s) + s z^{-\alpha}
    f'_{-1}(z^{-\alpha}s)].
  \end{align*}
  Make change of variable $s=\rx w$.  Then
  \begin{align*}
    \frac{b(y-z,y,\rx)}{\rx}
    &=
    \alpha z^{-\alpha-1}\int^1_0
      g_{\rx(1-w)}(y)[f_{-1}(\rx z^{-\alpha} w) +
    \rx z^{-\alpha} w f'_{-1}(\rx z^{-\alpha} w)]\,\dd w
    \\
    &\quad-
    z^{-\alpha} \int^1_0
      g'_{\rx(1-w)}(y) f_{-1}(\rx z^{-\alpha} w)\, \dd w.
  \end{align*}
  Put $u = \rx^{-\ra} y$ and $v = \rx^{-1} z^\alpha$.  Then
  $g_{\rx(1-w)}(y) = \rx^{-\ra} g_{1-w}(\rx^{-\ra} y) =
  \rx^{-\ra} g_{1-w}(u)$ and $g'_{\rx(1-w)}(y) = \rx^{-2\da}
  g'_{1-w}(\rx^{-\ra} y) = \rx^{-2\da} g'_{1-w}(u)$, and so
  \begin{align*}
    \frac{b(y-z,y,\rx)}{\rx}
    &=
    \alpha \rx^{-\ra} z^{-\alpha-1}
    \int^1_0 g_{1-w}(u)[f_{-1}(w/v) + (w/v) f'_{-1}(w/v)]\,\dd w
    \\
    &\quad-
    \rx^{-2/\rx} z^{-\alpha} \int^1_0 g'_{1-w}(u) f_{-1}(w/v)\, \dd w
    \\
    &=\rx^{-2\da-1}v^{-1} I(u,v),
  \end{align*}
  where the second equality is obtained by replace $z$ with $(\rx
  v)^\ra$ and 
  \begin{align*}
    I(u,v)
    &=
    \alpha v^{-\ra}
    \int^1_0 g_{1-w}(u)[f_{-1}(w/v) + (w/v) f'_{-1}(w/v)]\,\dd w
    -\int^1_0 g'_{1-w}(u) f_{-1}(w/v)\, \dd w.
  \end{align*}
  Combined with \eqref{e:a-diff} and $\dd y\,\dd z = \rx^{2\da}
  \alpha^{-1} v^{\ra-1}\,\dd u\,\dd v$, the above display yields 
  \begin{align} \nonumber
    &\frac{a(x, c, r+\rx) - a(x,c,r)}{\rx}
    \\\nonumber
    &= \rx^{-2\da-1} 
    \int_{y>0,z>0} h_c(c-y, r) \times
    v^{-1} I(u, v) h_z(x-c+z, 1 - r -\rx)\,\dd y\,\dd z
    \\\label{e:diff-a}
    &=\nth\alpha
    \int_{u>0,v>0} \frac{h_c(c-\rx^\ra u, r)}{\rx^\ra u}
    \times (u/v) I(u,v) \times
    \frac{h_{(\rx v)^\ra} ((\rx v)^\ra - (c-x), 1
      -r-\rx)}
    {(\rx v)^{1-\ra}} \,\dd u\,\dd v.
  \end{align}
  We need the following two lemmas.
  \begin{lemma}\label{l:integral}
    The function $(u/v) I(u,v)$ is integrable over $u>0$ and $v>0$
    with
    \begin{align*}
      \int_{u>0,v>0} (u/v) I(u,v)\,\dd u\,\dd v =
      \Gamma(\alpha+1).
    \end{align*}
  \end{lemma}

  \begin{lemma}\label{l:hc}
    The following statements hold.
    \begin{thmenum}
    \item \label{i:hc-x}
      Given $c>0$ and $t>0$, $h_c(c-x, t)/x$ as a function of $x$ is
      bounded on $(0,\infty)$ and
      \begin{align*}
        \lim_{x\to 0+} \frac{h_c(c-x,t)}{x} = m(c,t).
      \end{align*}
      Furthermore, given $c>0$, $m(c,t)$ is bounded in $t>0$.

    \item \label{i:hc-c}
      Given $x>0$ and $t>0$, $h_c(c-x,t)/c^{\alpha-1}$ as a function
      of $c$ is bounded on $(0,\infty)$ and
      \begin{align*} 
        \lim_{c\to0}\frac{h_c(c-x,t)}{c^{\alpha-1}}
        =\frac{f_{-x}(t)}{\Gamma(\alpha)}.
      \end{align*}
    \end{thmenum}
  \end{lemma}

  Assuming the lemmas are true, let $\rx\to0$ in \eqref{e:diff-a}.  By
  the lemmas and dominated convergence, the limit is $m(c,r)
  f_{x-c}(1-r)$.  With similar argument, $[a(x, c, r) - a(x, c,
  r-\rx)]/\rx$ converges to the same limit as $\rx\to0$.  Then
  \eqref{e:running-max} is proved. 

  To prove the rest of the corollary, integrate \eqref{e:running-max}
  over $x<c$.  From the identity $\intzi f_{-x}(s)\,\dd x =
  s^{\ra-1}/\Gamma(\ra)$ (\cite{sato:99:cup}, p.~270), it follows that
  $\rsup G_1$ and $\rsup X_1$ have joint \pdf
  \begin{align} \label{e:sup-joint}
    \frac{\pr{\rsup G_1\in\dd r, \rsup X_1\in\dd c}}{\dd r\,\dd c}
    =
    \frac{m(c,r)(1-r)^{\ra-1}}{\Gamma(\ra)}.
  \end{align}

  The conditional independence of $\rsup X_1$ and $\rsup X_1 - X_1$
  given $\rsup G_1$ follows from \eqref{e:running-max}.  As noted
  earlier, $\rsup G_1$ follows the $\Dbeta(1-\ra, \ra)$ distribution.
  This can be directly proved by integrating the above joint \pdf over
  $c>0$.  Since by \eqref{e:hc-y2},
  \begin{align*}
    \intzi m(c,r)\,\dd c
    =
    \frac{\sin(\pi\da)}{\pi} \intzi s^\ra
    E_{\alpha,\alpha}(-s)
    \Sbr{\intzi c^{-2} e^{-s r/c^\alpha}\,\dd c}\dd s\propto
    r^{-\ra},
  \end{align*}
  the \pdf of $\rsup G_1$ is in proportion to $r^{-\ra} (1-r)^{\ra -
    1}$, so it must be $\Dbeta(1-\ra, \ra)$.  Then conditionally on
  $\rsup G_1$, the \pdf of $\rsup X_1$ follows by dividing the
  joint \pdf of $\rsup X_1$ and $\rsup G_1$ by the \pdf of $\rsup
  G_1$, and the \pdf of $\rsup X_1-X_1$ follows from integrating
  $f_{x-c}(1-r)$ over $x<c$ and Kendall's identity \eqref{e:Kendall}.
\end{proof}

\begin{proof}[Proof of Lemma \ref{l:integral}]
  From the definition of $I(u,v)$, to show the integrability of $(u/v)
  I(u,v)$, it suffices to show
  \begin{align*}
    I_1&:=\int_{u,v>0} u v^{-1-\ra}
    \Sbr{\int^1_0 g_{1-w}(u) f_{-1}(w/v)\,\dd w}\,\dd u\,\dd v
    <\infty,
    \\
    I_2&:=\int_{u,v>0} u v^{-2-\ra}
    \Sbr{\int^1_0 w g_{1-w}(u) |f'_{-1}(w/v)|\,\dd w}\,\dd u\,\dd v
    <\infty,
    \\
    I_3&=\int_{u,v>0} u v^{-1}
    \Sbr{\int^1_0 |g'_{1-w}(u)| f_{-1}(w/v)\,\dd w}\,\dd u\,\dd
    v<\infty.
  \end{align*}
  By \eqref{e:moment-g}, for $w\in (0,1)$, $\int_{u>0} u
  g_{1-w}(u)\,\dd u = \mean[X_{1-w}\vee 0] =
  (1-w)^\ra/\Gamma(\ra)$.  Then by Fubini's theorem and
  \eqref{e:moment-f},
  \begin{align*}
    I_1
    &=
      \nth{\Gamma(\ra)}
      \int^1_0 (1-w)^\ra
      \Sbr{\intzi v^{-1-\ra} f_{-1}(w/v)\,\dd
      v}\,\dd w
    \\
    &= \nth{\Gamma(\ra)}
      \int^1_0 w^{-\ra} (1-w)^\ra
      \Sbr{\intzi t^{\ra-1} f_{-1}(t)\,\dd t}\dd w
    =\frac{\Gamma(\alpha)}{\alpha-1}.
  \end{align*}
  Similarly,
  \begin{align*}
    I_2
    &=
      \nth{\Gamma(\ra)}
      \int^1_0 w (1-w)^\ra \Sbr{
      \intzi v^{-2-\ra} |f'_{-1}(w/v)|\,\dd v}\,\dd w
    \\
    &=
      \nth{\Gamma(\ra)}
      \int^1_0 w^{-\ra} (1-w)^\ra \Sbr{
      \intzi t^\ra |f'_{-1}(t)|\,\dd t}\dd w <\infty.
  \end{align*}
  It follows that
  \begin{align*}
    \tilde I_2
    &:=
      \int_{u,v>0} u v^{-2-\ra}
      \Sbr{\int^1_0 w g_{1-w}(u) f'_{-1}(w/v)\,\dd w}\,\dd u\,\dd v
    \\
    &=
    \nth{\Gamma(\ra)}
    \int^1_0 w^{-\ra} (1-w)^\ra \Sbr{
      \intzi t^\ra f'_{-1}(t)\,\dd t
    }\dd w
    \\
    &=
    -\nth{\alpha\Gamma(\ra)}
    \int^1_0 w^{-\ra} (1-w)^\ra \Sbr{
      \intzi t^{\ra-1} f_{-1}(t)\,\dd t
    }\dd w = -\frac{\Gamma(\alpha)}{\alpha(\alpha-1)}.
  \end{align*}
  Next, since $C:=\intzi u|g'_1(u)|\,\dd u<\infty$ and $g'_{1-w}(u) =
  (1-w)^{-2\da} g'_1((1-w)^{-\ra} u)$,
  \begin{align*}
    I_3
    &= \int^1_0 \Cbr{\intzi v^{-1} f_{-1}(w/v)
    \Sbr{\intzi u|g'_{1-w}(u)|\,\dd u}\dd v} \dd w
    \\
    & = C\int^1_0 \Cbr{\intzi v^{-1} f_{-1}(w/v)
      \dd v} \dd w 
    = C \Gamma(1+\alpha) < \infty
  \end{align*}
  and 
  \begin{align*}
    \tilde I_3
    &:= \int^1_0 \Cbr{\intzi v^{-1} f_{-1}(w/v)
    \Sbr{\intzi ug'_{1-w}(u)\,\dd u}\dd v} \dd w
    \\
    &=
    \Gamma(\alpha+1) \intzi u g'_1(u)\,\dd u = -\Gamma(\alpha)
    (\alpha-1).
  \end{align*}
  Since $\int (u/v) I(u,v)\,\dd u\,\dd v = \alpha (I_1 + \tilde I_2) -
  \tilde I_3$, the proof then follows.
\end{proof}

\begin{proof}[Proof of Lemma \ref{l:hc}]
  \ref{i:hc-x} From the absolute convergence of the series
  \eqref{e:h-power-c}, as $x\to0+$, $h_c(c-x,t)/x$ converges to the
  limit with the expression \eqref{e:hc-y1}, which is  $m(c,t)$.  To
  show that $m(c,t)$ has the integral expression \eqref{e:hc-y2},
  first prove it for $n(t):=m(1,t)  = \lim_{x\to0+} [h_1(1-x,t)/x]$
  using the integral representation in Corollary \ref{c:h}, and then
  prove it in general using scaling.  Finally, by (21) on p.~210 of
  \cite{erdelyi:55:mcgraw}, $E_{\alpha, \alpha}(-s) = O(s^{-2})$ as
  $s\toi$.  Then given $t>0$,
  \begin{align*}
    |n(t)|\le \pi^{-1}\intzi s^\ra |E_{\alpha,
    \alpha}(-s)|\,\dd s<\infty,
  \end{align*}
  so $h_1(1-x,t)/x$ is bounded for $x\in (0,x_0)$ for small enough
  $x_0>0$.  On the other hand, by \eqref{e:pre-passage2} $h_1(1-x,t) <
  g_t(1-x)$, so $h_1(1-x,t)/x$ is bounded on $[x_0, \infty)$.  Thus
  $h_1(1-x, t)/x$ is bounded on $(0,\infty)$.  Furthermore, from
  \eqref{e:hc-y1}, $n(\cdot)$ has an analytic extension
  to $\{z\in\Coms: \Re z>0\}$.  As a result, $n(t)>0$ for almost every
  $t>0$ under the Lebesgue measure.  Fix $r\in (0,t)$.  By the Markov
  property, for any $x>0$,
  \begin{align*}
    h_1(1-x, t)
    &=
    \int_{u>0} \pr{X_r\in 1-\dd u, \rsup X_r<1}  h_u(u-x, t-r)
    \\
    &=
    \int_{u>0} h_1(1-u, r) h_u(u-x, t-r)\,\dd u.
  \end{align*}
  Divide both sides by $x$ and let $x\to0+$.  By Fatou's lemma and
  $m(c,t) = n(t/c^\alpha)/c^2$,
  \begin{align*}
    n(t) \ge \int_{u>0} h_1(1-u,r) m(u, t-r)\,\dd u
    = \int_{u>0} h_1(1-u,r) \frac{n((t-r)/u^\alpha)} {u^2}
    \,\dd u.
  \end{align*}
  By Corollary \ref{c:nonzero}, $h_1(1-u,r)>0$ for all $u>0$.  Then the
  integral on the \rhs is positive, and so $n(t)>0$.

  \ref{i:hc-c} The convergence follows from \eqref{e:h-series-c}.
  That $h_c(c-x,t)/c^{\alpha-1}$ is a bounded function of $c$ on
  $(0,\infty)$ can be similarly proved as in \ref{i:hc-x}.
\end{proof}

\begin{remark}
  By duality, for $t=1$, the limit in Lemma \ref{l:hc}\ref{i:hc-x}
  can be written as
  \begin{align*}
    \frac{\pr{X_1\in \dd c - x\gv \rinf X_1 >-x}}{\dd c}\times
    \frac{\pr{\rinf X_1>-x}}{x} \to m(c,1), \quad
    x\to0.
  \end{align*}
  Since $\pr{\rinf X_1>-x} = \pr{\fht_{-x}>1} = \pr{\fht_{-1} >
    x^{-\alpha}}\sim x/\Gamma(1-\ra)$ as $x\to0$, then the
  display suggests that $\Gamma(1-\ra) m(c,1)$ can be regarded as
  the conditional \pdf of $X_1$ at $c>0$ given $\rinf X_1\ge 0$.
\end{remark}

\section{Exact sampling for first passage} \label{s:exact-sampling}
In this section, it will be shown that it is possible to conduct exact
joint sampling of $\fpt_c$, $X_{\fpt_c-}$, and $\jp_{\fpt_c}$ for
a spectrally positive stable $X$ satisfying \eqref {e:stable-normal}.
From Proposition \ref{p:fp-df}, this may be done in two steps.  The
first step is to jointly sample $X_{\fpt_c-}$ and $\jp_{\fpt_c}$,
which is standard.  The second step is to sample $\fpt_c$ given 
$X_{\fpt_c-}$, which is the focus of the section.  Since by scaling,
$(\fpt_c, X_{\fpt_c-}, \jp_{\fpt_c}) \sim (c^\alpha \fpt_1,
c X_{\fpt_1-}, c \jp_{\fpt_1})$, it suffices to consider $c=1$.

\subsection{Sampling of pre-passage value and jump}
Because $\pr{X_1>0} = 1 - \ra$ (\cite{bertoin:96:cup}, p.~218),
from Example 7 in \cite{doney:06:aap}, at every $(x,z,w)$, the joint
\pdf of $X_{\fpt_1-}$, $\jp _{\fpt_1}$, and $\rsup X_{\fpt_1-}$ takes
value $C\cf{x\vee 0\le w < 1, z > 1-x>0} w^{\alpha-2} z^{-1-\alpha}$,
where $C=\pi^{-1}\alpha(\alpha-1) \sin((\alpha - 1)\pi)$.  It follows
that $X_{\fpt_1-}\sim\xi$, where $\xi$ has \pdf
\begin{align*}
  p(x) = C' \cf{x<1} [1 - (x\vee 0)^{\alpha - 1}](1 - x)^{-\alpha}
\end{align*}
with $C'>0$ a constant, and for every $x<1$, conditionally on
$X_{\fpt_1-}=x$, $\jp _{\fpt_1}\sim (1-x) \zeta$, where $\zeta$
has \pdf $q(z) = \alpha\cf{z>1} z^{-\alpha-1}$.  Thus the joint
sampling of $X_{\fpt_1-}$ and $\jp_{\fpt_1}$ boils down to that
of independent $\xi\sim p$ and $\zeta\sim q$.  The sampling
of $\zeta$ is straightforward as $\zeta \sim U^{-\ra}$,
where $U\sim \Dunif(0,1)$.  To sample $\xi$, it can be seen that $p(x)
= \theta p_1(x) + (1-\theta) p_2$, where $\theta =
m_1/(m_1+m_2)$ with $m_1 = (\alpha-1)^{-1}$, $m_2 =
\lfrac{\pi}{\sin((\alpha-1)\pi)}  - (\alpha-1)^{-1}$, and 
\begin{align*}
  p_1(x) = \cf{x\le 0} (1-x)^{-\alpha}/m_1, \quad
  p_2(x) = \cf{0< x<1} (1-x^{\alpha-1}) (1-x)^{-\alpha}/m_2
\end{align*}
are two \pdf's.  On one hand, $p_1(x)$ is the \pdf of
$1-U^{-1/(\alpha-1)}$.  On the other, $p_2(x) \propto
\cf{0<x<1}(1-x^{\alpha-1}) (1-x)^{-\alpha} < \rho(x):=\cf{0<x<1}
(1-x)^{-\alpha+1}$.  Using the fact that $\rho(x)$ is proportional to
the \pdf of $1-U^{1/(2-\alpha)}$, $p_2$ can be sampled by the
rejection sampling method (\cite{devroye:86:sv-ny}, Chapter II).  In
summary, $p(x)$ can be sampled as follows. 
\begin{thmenum}
\item Sample $I$ from $\{1,2\}$ such that $\pr{I=1} = m_1/(m_1+m_2)$
\item If $I=1$, then sample $U\sim \Dunif(0,1)$ and output
  $1-U^{-1/(\alpha -1)}$, otherwise, do the following iteration until
  an output is made.
  \begin{itemize}[leftmargin=2.5ex]
  \item 
    Sample $U$, $V$ \iid $\sim\Dunif(0,1)$ and set $x = 1 -
    U^{1/(2-\alpha)}$.  If $V\le (1 - x^{\alpha-1})/(1-x)$, then
    output $x$, otherwise repeat.
  \end{itemize}
\end{thmenum}

\subsection{Sampling of time of first passage}
We now consider the sampling of $\fpt_1$ conditionally on $X_{\fpt_1-}
= x\in (-\infty, 1)$.  By Proposition \ref{p:pre-passage}, if $x<0$,
then $h_1(x, \cdot)/v_1(x)$ is the \pdf of $\tau' + \xi$, with $\tau'
\sim h_1(0, \cdot)/v_1(0)$ and $\xi\sim f_x$ being independent.  Since
the sampling of $\xi$ is well known \cite{chambers:76:jasa}, the
sampling of $h_1(x, \cdot)/v_1(x)$ can be reduced to that of $h_1(0,
\cdot)/v_1(0)$.  As a result, it only remains to consider the case
$0\le x < 1$. 

We again will use the rejection sampling method.  For this method, the
normalizing constant $v_c(x)$ is not important and one can just focus
on $h_1(x,\cdot)$.  We will use the the power series representation
\eqref{e:h-power} of $h_1(x,\cdot)$.  In order to handle the infinite
number of positive and negative terms in the series, we first describe
the general approach to use.

Let $p$ and $q$ be two \pdf's that are proportional to some explicit
functions $f$ and $g$, respectively, whose normalizing constants may
be intractable; $g$ is known as an envelope function.  For the
rejection sampling method, $q$ must be easy to sample.  Suppose $f$
can be decomposed as
\begin{align} \label{e:f-decomp}
  \begin{array}{c}
    \displaystyle
    f(t) = \sumoi l \phi_l(t) \quad\text{such that for some explicit
      constants $\inum c$}\\
    \displaystyle
    0\le \phi_l(t) \le c_l g(t) \quad\text{with }\ 
    C := \sum c_l < \infty.
  \end{array}
\end{align}
Then $p$ can be sampled as follows.
\begin{itemize}
\item Independently sample $T\sim q$, $U\sim \Dunif(0,1)$, and $\ell$
  from the probability mass function $\pr{\ell = l} = c_l/C$.  If
  $U\le \phi_\ell(T)/(c_\ell g(T))$, then output $T$ and stop,
  otherwise repeat.
\end{itemize}
Indeed, by standard argument of the rejection sampling method, the
\pdf of the output of the procedure is proportional to
\begin{align*}
  g(t) \sum_l \Sbr{\frac{c_l}{C} \times \frac{\phi_l(t)}{c_l g(t)}}
  =
  \sum_l \phi_l(t)/C = f(t)/C,
\end{align*}
so it must be $p$.  The point is that when $f(t)$ is an infinite
series that cannot be evaluated in closed form, say $f(t) = \sum_{a\in
  A} f_a(t)$, it is possible to have each $\phi_l(t)$ equal to the sum
of a finite set of $f_a(t)$.  More precisely, $\phi_l(t) = \sum_{a\in
  A_l(t)} f_a(t)$, where $A_l(t)$ is a finite subset of $A$ that may
depend on $t$, and given $t$, $A_1(t)$,  $A_2(t)$, \ldots, form a
partition of $A$.  It is also critical the $A_l(t)$'s are such that
$\phi_l(t)\ge0$ for all $l$ and $t$.  In each iteration, once $T$ and
$\ell$ are sampled, only $\phi_l(T)$ with $l$ equal to the value of
$\ell$ needs to be evaluated.  As long as for any $t$, each $f_a(t)$
is easy to evaluate, and the set $A_l(t)$ can be enumerated in a
finite number of steps, $\phi_l(t)$ can be evaluated exactly.

To apply the above approach to $h_1(x, t)$, where $x<1$ is fixed, the
main issue is the construction of the envelop function and the
$\phi_l(t)$'s.  The next lemma gives an option for the envelope
function.
\begin{lemma} \label{l:h-dominate}
  Fixing any $D\ge \sup_{n\ge 1} 2^{n-1}\Gamma(n)/\Gamma(\alpha n)$,
  define
  \begin{align*}
    \theta = 4^{1/(\alpha-1)}, \quad
    C_\alpha = (\alpha\Gamma(1-\ra))^{-1} \vee [D(\theta^\alpha
    e^\theta  + 4)], \quad
    H_\alpha(t) = C_\alpha t^{-\ra} \wedge t^{-1-\alpha}, \quad
    t>0.
  \end{align*}
  Then for every $0\le x<1$ and $t>0$,  $h_1(x,t) \le H_\alpha(t)$. 
\end{lemma}

The normalized $H_\alpha(t)$ is $\theta p_1(t) +(1-\theta)
p_2(t)$, where $p_1(t)= (1-\ra) \cf{0<t<1} t^{-\ra}$ and
$p_2(t) = \alpha \cf{t>1} t^{-\alpha-1}$ are \pdf's and $\theta =
\alpha^2/(\alpha^2 + \alpha-1)$.  Thus the normalized $H_\alpha$ can
be sampled as follows.
\begin{itemize}[topsep=1ex, itemsep=0ex]
\item Sample $U$, $V$ \iid $\sim\Dunif(0,1)$.  If $U\le \theta$,
  return $V^{\alpha/(\alpha-1)}$, otherwise return $V^{-\ra}$.
\end{itemize}
As a result, $H_\alpha$ can be used as an envelope function.

Now consider the construction of $\phi_l(t)$.  Let $c_l = 2^{-l+1}$.
Then from \eqref{e:f-decomp}, we wish to construct 
$0 < \phi_l(t) < 2^{-l+1} H_\alpha(t)$ such that
$h_1(x,t) = \sumoi l \phi_l(t)$.  Write
\begin{align*}
  m_{k,n}(s,u)
  =
  \frac{\Gamma(k\da + n) s^k u^n}{\pi k!\Gamma(\alpha n)},
\end{align*}
so that
\begin{align*}
  h_1(x,t)
  =
  \sumoi {k,n} \sin (\pi k\da) m_{k,n}(-(1-x) t^{-\ra},
  -t^{-1}).
\end{align*}

We shall construct for each $t>0$ a sequence of finite sets
$\Lambda_l(t) \subset \Nats\times\Nats$, $l\ge 0$, such that
$\Lambda_l(t)\subset \Lambda_{l+1}(t)$, $\bigcup^\infty_{l=1}
\Lambda_l(t) = \Nats\times\Nats$ and 
\begin{align*}
  F_l(t) :=
  \sum_{(k,n)\in\Lambda_l(t)}
  (-1)^{k+n} \sin (\pi k\da)
  M_{k,n}
\end{align*}
is strictly increasing in $l$ such that $0 < h_1(x,t) - F_l(t) \le
2^{-l} H_\alpha(t)$, where $M_{k,n} = m_{k,n}(s,u)$ with $s = (1-x)
t^{-\ra}$ and $u = t^{-1}$.   Once this is done, let $\phi_l(t) =
F_l(t) - F_{l-1}(t)$.   Then $\sumzi l \phi_l(t) = \lim_l F_l(t) =
h_1(x,t)$ and $0 < \phi_l(t) < h_1(x,t) - F_{l-1}(t) \le 2^{-l+1}
H_\alpha(t)$, as desired.  The construction is based on the following
two lemmas.
\begin{lemma} \label{l:sum-bound}
  Fix $\rx\in (0,1/2)$ and $s$, $u>0$.  Let $k$ and $n\in\Nats$ such
  that $n>(2 u/\rx)^{1/(\alpha-1)}$, $k>(2 s/\rx)^{\alpha / (\alpha
    -1)}$, and $k\da \le n\le (2-\ra) k$, then
  \begin{align} \label{e:S1}
    \sum^\infty_{i,j=0, i+j\ge 1} m_{k+i, n+j}(s,u)
    \le 24\rx m_{k,n}(s,u),
    \\\label{e:S2}
    \sumoi j m_{k',n+j}(s,u) \le 2\rx m_{k', n}(s,u)
    \quad\forall k'\le k,
    \\\label{e:S3}
    \sumoi i m_{k+i, n'}(s,u) \le 2\rx m_{k, n'}(s,u)
    \quad\forall n'\le n.
  \end{align}
\end{lemma}

\begin{lemma} \label{l:pair}
  Let $d_\alpha = (\ra-1/2)\wedge [1/2 - 1/(2\alpha)]$ and
  $L_\alpha = \Flr{(\alpha - 1/2)/(\alpha - 1)}+1 \ge 2$.  Then among
  any $2 L_\alpha$ consecutive integers, there exist an even number
  and an odd number both belonging to $A_\alpha := \cup_{j\in\Ints}
  I_j$, where $I_j = [(2j+d_\alpha)\alpha, (2j+1-d_\alpha)\alpha]$.
\end{lemma}

Assume the two lemmas are true for now.  Let $\Lambda_0(t)
= \emptyset$ and $F_0(t)=0$.  By Corollary \ref{c:nonzero} and Lemma
\ref{l:h-dominate}, $0<h_1(x,t) - F_0(t) = h_1(x,t)  \le H_\alpha(t)$.
Suppose $\Lambda_l(t)$ has been constructed, such that $F_l(t)\ge 0$
and $0<h_1(x,t) - F_l(t) \le 2^{-l+1} H_\alpha(t)$.  We need to
construct $\Lambda_{l+1}(t)\supset \Lambda_l(t)$, such that
$F_{l+1}(t) > F_l(t)$ and $0<h_1(x,t) - F_{l+1}(t) \le 2^{-l}
H_\alpha(t)$.

For $r\in\Nats$, denote $S_r = \{(k,n):\, k,n=1,\ldots, r\}$ and 
$\partial S_r = \{(k,n)\in S_r: k \vee n=r\}$ its ``boundary''.  
Let $d_\alpha$ and $A_\alpha$ be as in Lemma \ref{l:pair}.  Let
$\delta_\alpha = \sin(d_\alpha \pi)$ and $K_\alpha = \Ints\cap
A_\alpha$.  Then $\delta_\alpha>0$ and for $k\in K_\alpha$, $\pi
k\da\in [(2j+d_\alpha)\pi, (2j+1-d_\alpha)\pi]$ for some
$j\in\Ints$, so for $n$ of the same parity as $k$,
\begin{align*}
  (-1)^{k+n} \sin(k\pi\da) = \sin(k\pi\da) \ge 
  \delta_\alpha>0.
\end{align*}

Put $\rx = \delta_\alpha / 24$.  Let $R$ be the smallest integer such
that
\begin{align*}
  R>
  \frac{2 L_\alpha}{\alpha-1} \vee 
  \Grp{\frac{2  u}\rx}^{\lfrac1{(\alpha-1)}} \vee
  \Grp{\frac{2 s}\rx}^{\lfrac\alpha{(\alpha-1)}},\quad 
  \Lambda_l(t)\subset S_R,\quad
  \sum_{(k,n)\in \partial S_R} M_{k,n}
  \le \frac{2^{-l} H_\alpha(t)}{24\rx}.
\end{align*}
Starting with $r=R$, do the following iteration.
\begin{itemize}
\item For each $n$, let $k_n$ be the smallest number in $K_\alpha\cap
  [r+1,\infty)$ that has the same parity as $n$; $k_n$ exists because by
  Lemma \ref{l:pair}, $K_\alpha\cap \{r+1,\ldots, r+2L_\alpha\}$
  contains an even number and an odd number.  In particular, $1\le k_n
  - r \le 2 L_\alpha$.  Define 
  \begin{align*}
    S''_r = S_r \cup\bigcup^r_{n=1} \{(k,n): r<k<k_n\}
  \end{align*}
  and $S'_r = S''_r \cup \{(k, r+1): (k,r)\in S''_r, (-1)^{k+r}
  \sin(\pi k\da) > 0\}$.  If
  \begin{align*}
    \sum_{(k,n)\in S'_r} (-1)^{k+n} \sin(\pi k\da) M_{k,n}
    > F_l(t),
  \end{align*}
  then let $\Lambda_{l+1}(t) = S'_r$ and stop.  Otherwise increase $r$
  by 1 and repeat.
\end{itemize}

Since $h_1(x,t) - F_l(t)>0$ and $\sum_{(k,n)\in S'_r} (-1)^{k+n}
\sin(\pi k\da) M_{k,n}\to h_1(x,t)$ as $r\toi$, the iteration
eventually will stop.  It is clear that
$\Lambda_{l+1}(t) = S_r \supset S_R \supset \Lambda_l(t)$ and
$F_{l+1}(t) > F_l(t)$.  Next,
\begin{align*}
  h_1(x,t) - F_{l+1}(t)
  &=
  \sum_{(k,n)\not\in S'_r} (-1)^{k+n} \sin(\pi k\da) M_{k,n}
  \le
  \sum_{(k,n)\not\in S_r} M_{k,n}
  \le 
  \sum_{(k,n)\not\in S_R} M_{k,n}.
\end{align*}
Since $R > (2u/\rx)^{1/(\alpha-1)}\vee(2 s/\rx)^{\alpha/(\alpha-1)}$,
by Lemma \ref{l:sum-bound},
\begin{align*}
  \sum_{(k,n)\not\in S_R} M_{k,n}
  &=
  \sum^\infty_{i,j=0,i+j\ge 1} M_{R+i,R+j}
  + \sum^{R-1}_{k=1} \sumoi j M_{k,R+j}
  + \sum^{R-1}_{n=1} \sumoi i M_{R+i, n}
  \\
  &\le
  24 \rx M_{R,R} + 2\rx \sum^{R-1}_{k=1} M_{k,R} +
  2\rx \sum^{R-1}_{n=1} M_{R,n} \le 24\rx \sum_{(k,n)\in \partial S_R}
  M_{k,n}.
\end{align*}
By the choice of $R$, the above two displays give $h_1(x,t) -
F_{l+1}(t) < 2^{-l} H_\alpha(t)$.  It only remains to show 
$h_1(x,t) - F_{l+1}(t) > 0$, i.e. $\sum_{(k,n)\not\in S'_r} (-1)^{k+n}
\sin(\pi k\da) M_{k,n}>0$.   It can be seen that $(\Nats\times
\Nats) \setminus S'_r$ can be partitioned into  the following sets:
\begin{align*}
  E_1
  &=
  \{(k,n): k<k_r, (-1)^{k+r}\sin(k\pi\da)\le 0, n\ge r+1\},
  \\
  E_2
  &=
  \{(k,n): k<k_r, (-1)^{k+r} \sin(k\pi\da)>0, n\ge r+2\},
  \\
  E_3
  &=
  \{(k,n): k\ge k_n, n\le r-1\},
  \\
  E_4
  &=
  \{(k,n): k\ge k_r, n\ge r\}.
\end{align*}
As already seen, $1\le k_r - r\le 2 L_\alpha$.  Then 
$k_r\da \le r+1$ and $r+2\le (2-\ra) k_r$, the first one due to
$r - k_r\da \ge (1 - \ra) r - 2 L_\alpha\da \ge (1-\ra) R -
2L_\alpha\da > 0$ and the second one $(2-\ra)k_r - r -2 \ge
(2-\ra)(r+1) - r - 2\ge (1-\ra) R-\ra\ge 0$.
Also, $r> (2u/\rx)^{1/(\alpha-1)}$ and $k_r >
(2s/\rx)^{\alpha/(\alpha - 1)}$.  Then by \eqref{e:S2} in Lemma
\ref{l:sum-bound}, for every $k<k_r$, $\sumoi j M_{k, r+1+j} \le 2\rx
M_{k, r+1}$, giving
\begin{align*}
  \sum_{n:(k,n)\in E_1} (-1)^{k+n} \sin(k\pi\da) M_{k,n}
  &= 
  (-1)^{k+r+1} \sin(k\pi\da)
  \Sbr{M_{k,r+1} - \sum_{n\ge r+2} (-1)^{n-r-1} M_{k,n}}
  \\
  &\ge 
  |\sin(k\pi\da)|(1-2\rx) M_{k,r+1}\ge 0.
\end{align*}
Since $k_r\ge 2L_\alpha+1$, by Lemma \ref{l:pair}, $(-1)^{k+r+1}
\sin(k\pi\da)>0$ for at least one $k<k_r$.  Thus the sum over $E_1$ is
strictly positive.  Likewise, the sum  over $E_2$ is strictly
positive.  Next, the sum over $E_3$ is at least
\begin{align*}
  \sum^{r-1}_{n=1} \Grp{
    (-1)^{k_n+n} \sin(k_n\pi\da) M_{k_n, n}
    - \sumoi j M_{k_n,n+j}
  }
  \ge 
  \sum^{r-1}_{n=1} \Grp{
    \delta_0 M_{k_n, n}
    - \sumoi j M_{k_n,n+j}
  }.
\end{align*}
By \eqref{e:S2} in Lemma \ref{l:sum-bound}, the last sum is strictly
positive.  Similar, using \eqref{e:S1} in Lemma \ref{l:sum-bound}, the
sum over $E_4$ is strictly positive.  Thus $h_1(x,t) - F_{l+1}(t)>0$,
as desired.

\subsection{Proof of Lemmas}
\begin{proof}[Proof of Lemma \ref{l:h-dominate}]
  Given $x\in [0,1)$, by \eqref{e:pre-passage}, \eqref{e:pdf-g},
  and $g_t$ being decreasing on $[0,\infty)$ (\cite {sato:99:cup},
  p.~416), 
  \begin{align*}
    h_1(x,t) \le g_t(x) \le g_t(0)
    =
    t^{-\ra}/(\alpha \Gamma(1-\ra)), \quad t>0.
  \end{align*}
  On the other hand, for $t\ge 1$, from \eqref{e:h-series-bound}, 
  \begin{align*}
    h_1(x,t)
    \le
    t^{-\ra-1}
    \sumoi{k,n} \frac{\Gamma(k\da+n)}{k!\Gamma(\alpha n)}
    \le
    B t^{-\ra-1}, 
  \end{align*}
  where $B= \intzi E_{\alpha,\alpha}(s) e^{s^\ra-s}\,\dd s$.
  By $E_{\alpha,\alpha}(s) = \sumoi n \lfrac{s^{n-1}}{\Gamma(\alpha
    n)} \le D\sumoi n \lfrac{(s/2)^{n-1}} {\Gamma(n)} = D e^{s/2}$, $B
  \le D\intzi e^{s^\ra-s/2} \,\dd s \le
  D(\int^{\theta^\alpha}_0 e^{s^\ra}\,\dd s +
  \int^\infty_{\theta^\alpha} e^{-s/4}\,\dd s) \le D(\theta^\alpha
  e^\theta + 4)$, which together with the displays yields the proof.
\end{proof}

To prove Lemma \ref{l:sum-bound}, we need the following.
\begin{lemma} \label{l:Gautschi}
  Let $k$ and $n\in\Nats$, and $s$, $u>0$.
  \begin{thmenum}
  \item \label{i:G1}
    If $n \ge k\da$, then $2 u m_{k,n}(s,u) / (n-1)^{\alpha-1}>
    m_{k,n+1}(s,u)$. 
  \item \label{i:G2}
    If $n\le (2-\ra)(k+1)$, then $2 sm_{k,n}(s,u)/(k+1)^{1 -
      \ra}  > m_{k+1,n}(s,u)$.
  \item \label{i:G3}
    If $k\da \le n \le (2-\ra)(k+1)$, then $6 s um_{k,n}(s,u) /
    (n-1)^{\alpha - 1} (k+1)^{1-\ra} > m_{k+1,n+1}(s,u)$.
  \end{thmenum}
\end{lemma}
\begin{proof}
  \ref{i:G1} For $k\ge 1$ and $n\ge k\da$,
  \begin{align*}
    \frac{m_{k,n}(s,u)}{m_{k,n+1}(s,u)}
    =
    \frac{u^{-1}\Gamma(\alpha n + \alpha)}
    {\Gamma(\alpha n)(k\da + n)}
    \ge
    \frac{u^{-1}}{2 n} \frac{\Gamma(\alpha n + \alpha)}
    {\Gamma(\alpha n)}.
  \end{align*}
  By Gautschi's inequality (\cite{NIST:10}, p.~138),
  $\lfrac{\Gamma(\alpha n + \alpha)}{\Gamma(\alpha n)} > (\alpha n +
  \alpha - 1) (\alpha n + \alpha - 2)^{\alpha - 1} > n (n
  -1)^{\alpha-1}$,  which together with the display yields the proof.

  \ref{i:G2}   By Gautschi's inequality, $\Gamma((k+1)\da + n) <
  \Gamma(k \da + n) ((k+1)\da + n)^\ra$.  
  Then 
  \begin{align*}
    \frac{m_{k,n}(s,u)}{m_{k+1,n}(s,u)}
    =
    \frac{s^{-1}\Gamma(k\da + n) (k+1)}{\Gamma((k+1)\da + n)}
    >
    \frac{s^{-1}(k+1)}{((k+1)\da + n)^\ra}.
  \end{align*}
  If $n\le (2-\ra)(k+1)$, then $((k+1)\da+n)^\ra \le
  [2(k+1)]^\ra<2(k+1)^\ra$, leading to the proof.

  \ref{i:G3} From the above argument, if $n\ge k\da$, then
  \begin{align*}
    \frac{m_{k,n}(s,u)}{m_{k+1,n+1}(s,u)}
    =
    \frac{m_{k,n}(s,u)}{m_{k,n+1}(s,u)}
    \frac{m_{k,n+1}(s,u)}{m_{k+1,n+1}(s,u)}
    \ge
    \frac{(su)^{-1}(n-1)^{\alpha-1}}{2} \frac{(k+1)}{((k+1)\da +
      n+1)^\ra}.
  \end{align*}
  Then for $n\le (2-\ra)(k+1)$, $(k+1)\da + n+1 \le
  2(k+1)+1\le 3(k+1)$, so
  \begin{align*}
    \frac{(k+1)}{((k+1)\da +   n+1)^\ra}
    \ge
    \frac{(k+1)}{(3(k+1))^\ra},
  \end{align*}
  which together with the previous display yields the proof.
\end{proof}

\begin{proof}[Proof of Lemma \ref{l:sum-bound}]
  Write $m_{k,n} = m_{k,n}(s,u)$ and $S_{k,n} =
  \sumzi{i,j} m_{k+i, n+j}$.  Then \eqref{e:S1} is equivalent to
  $S_{k,n} \le (1+24\rx) m_{k,n}$ for $k$, $n$ satisfying the 
  conditions in the lemma.  Let $k_0=k$ and for $l\ge 1$, $k_l =
  \Flr{\alpha (n+l-1)+1}$.  Then by $\alpha \in (1,2)$, $k_0 < k_1 <
  k_2 < \ldots$ and $k_l\da\le (k_{l+1}-1)\da
    \le n + l \le (2-\ra) k_l$ for $l\ge 0$.  Put $d_l = k_{l+1}
    - k_l$.  Then
  \begin{align*}
    S_{k,n}
    &=
    \sumzi l \Grp{\sum^{d_l - 1}_{i=0} \sumoi j m_{k_l+i,
        n+l + j} + \sumzi i m_{k_l + i, n+l}
    }.
  \end{align*}
  For $0\le i<d_l$, and $j\ge 1$, since $n+l+j-1\ge  n+l\ge (k_{l+1} -
  1)\da \ge (k_l +  i)\da$, by Lemma \ref{l:Gautschi}\ref{i:G1},
  $m_{k_l+i, n+l+j}/m_{k_l+i, n+l+j-1} \le 2u (n+l+j-2)^{1-\alpha}
  \le 2u(n-1)^{1-\alpha} < \rx$.  Then by induction,
  \begin{align*}
    \sum^{d_l - 1}_{i=0} \sumoi j m_{k_l+i, n+l + j}
    <
    \sum^{d_l- 1}_{i=0} \sumoi j \rx^j m_{k_l+i, n+l}
    \le 
    \frac{\rx}{1-\rx} \sumzi i m_{k_l + i,  n+l}
  \end{align*}
  and hence
  \begin{align*}
    S_{k,n} < \nth{1-\rx} \sumzi l \sumzi i m_{k_l + i, n+l}.
  \end{align*}
  For each $i\ge 1$, since $n+l\le (2-\ra) (k_l + i)$, by  Lemma
  \ref{l:Gautschi}\ref{i:G2}, $m_{k_l+i,  n+l}/m_{k_l+i-1, n+l}\le 2s
  (k_l+i)^{\ra - 1} < 2 s k^{\ra - 1} < \rx$.  Then by induction,
  $m_{k_l+i, n+l} \le \rx^i m_{k_l, n+l}$, resulting in
  \begin{align*}
    S_{k,n}< \nth{(1-\rx)^2} \sumzi l m_{k_l, n+l}.
  \end{align*}
  For each $l\ge 1$, since $(k_l-1)\da \le n+l-1 \le (2 -
  \ra)(k_l-1)$, by Lemma \ref{l:Gautschi}\ref{i:G3} $m_{k_l,
    n+l}/ m_{k_{l-1}, n+l-1}< 6su(n-1)^{1-\alpha} k^{\ra-1} <
  \rx$.  Then by induction, $S_{k,n}< (1-\rx)^{-3} m_{k,n} <
  (1+24\rx) m_{k,n}$, as desired.  The proof for \eqref{e:S2} and
  \eqref{e:S3} is very similar to that for \eqref{e:S1} and hence is
  omitted.
\end{proof}

\begin{proof}[Proof of Lemma \ref{l:pair}]
  Recall that $I_j$ is defined to be $[(2j+d_\alpha)\alpha,
  (2j+1-d_\alpha)\alpha]$.  Then $|I_j|\ge 1$.  Let $B_j =
  ((2j+1-d_\alpha) \alpha, (2j+2+d_\alpha)\alpha)$.  Then $|B_j|\le2$
  and $A^c_\alpha = \cup_{j\in\Ints} B_j$.  If two consecutive
  integers both belong to $A^c_\alpha$, they must
  belong to the same $B_j$, for otherwise there would be an $I_i$
  strictly between the two, implying $|I_i|<1$.  Moreover, no three
  consecutive integers can all belong to $A^c_\alpha$, for otherwise
  they had to be in the same $B_j$, implying $|B_j|>2$.   Assume that
  for some $i$, none of the even numbers in $S = \{i+1, i+2, \ldots,
  i+2L_\alpha\}$ is in $K_\alpha$.  Then all the odd numbers in $S$
  are in $K_\alpha$.  Consequently, the even numbers belong to
  $L_\alpha$ different $B_j$'s, and the odd ones to $L_\alpha$
  different $I_j$'s.  The union of these intervals  has Lebesgue
  measure $2\alpha L_\alpha$.  Since the union lies between $i+1-|C|$
  and $i+2L_\alpha+|D|$, where $C$ is the interval containing $i+1$
  and $D$ the one containing $i+2L_\alpha$, then $2\alpha L_\alpha \le
  L_\alpha-1+|C| + |D|$.  Observe that either $C$ is an $I_j$ and $D$
  is a $B_l$, or vice versa.  Then $|C|+|D| = 2\alpha$, so $2\alpha
  L_\alpha \le 2L_\alpha -1 +  2\alpha$, contradicting the choice for
  $L_\alpha$.  This shows there is at least one even number in $S$
  belonging to $K_\alpha$.  Likewise, there is at least one odd number
  in $S$ belonging to $K_\alpha$.
\end{proof}

\bibliographystyle{acmtrans-ims}

\section*{Appendix}
\begin{proof}[On the connection between \eqref{e:h-diff} and
  \eqref{e:Brownian}] 
  When $\alpha=2$, $\sin(\pi k\da)$ is 0 if $k$ is even and is
  $(-1)^j$ is $k = 2j + 1$ for integer $j\ge 0$.  Then
  the series in \eqref{e:h-diff} can be written as
  \begin{align*}
    \nth{\pi}
    \sumoi {j=0,n} \frac{\Gamma(j+1/2+n)}{(2j+1)!(2n-1)!}
    (-1)^{j+n+1} [(2j+1)c + (2n-1)(c-x)](c-x)^{2j} c^{2n-2}.
  \end{align*}
  Write $n=l+1$ and $m=j+l$.  Then the series becomes
  \begin{align*}
    &
    \nth{\pi}
    \sumzi {j,l} \frac{\Gamma(j + l+3/2)}{(2j+1)!(2l+1)!}
    (-1)^{j+l} [(2j+1)(c-x)^{2j} c^{2l+1} + (2l+1) (c-x)^{2j+1}
    c^{2l}]
    \\
    &=
    \nth{\pi}
    \sumzi m \frac{\Gamma(m+3/2)}{(2m+1)!} (-1)^m 
    \sum^m_{j=0}\Sbr{
      \frac{(c-x)^{2j} c^{2m-2j+1}}{(2j)!(2m-2j+1)!}
      +
      \frac{(c-x)^{2j+1} c^{2m-2j}}{(2j+1)!(2m-2j)!}
    }
    \\
    &=
    \nth{\pi}
    \sumzi m \frac{\Gamma(m+3/2)}{(2m+1)!} (-1)^m 
    \sum^{2m+1}_{s=0} \frac{(2m+1)!}{s!(2m+1-s)!} (c-x)^s c^{2m+1-s}
    \\
    &=
    \nth{\pi}
    \sumzi m \frac{\sqrt\pi}{2^{2m+1} m!}(-1)^m 
    (2c-x)^{2m+1}
    \\
    &=
    \frac{2c-x}{2\sqrt{\pi}}\exp\Cbr{-\frac{(2c-x)^2}{4}}.
  \end{align*}
  Since $(X_t)_{t\ge 0}\sim (W_{2t})_{t\ge 0}$, this is
  essentially the same result as \eqref{e:Brownian}.
\end{proof}

\begin{proof}[Proof of Eq.~\eqref{e:ladder}]
  We need the following refined version of Lemma
  \ref{l:hc}\ref{i:hc-x}.
  \begin{lemma} \label{l:hc2}
    There is a constant $M>0$, such that for all $0<x<1/2$ and all
    $t>0$,
    \begin{align*}
      h_1(1-x,t) \le M x(t+t^{1-\ra}).
    \end{align*}
  \end{lemma}
  Assume the lemma is true for now.  Then given $c>0$, by scaling,
  for all $0<x<c/2$, $h_c(c-x, t)\le M x(t+t^{1-\ra})$ for some $M
  = M(c)>0$.  Then by Lemma \ref{l:hc}\ref{i:hc-x} and dominated
  convergence, for each $q>0$,
  \begin{align*}
    \intzi m(c,t) e^{-q t}\,\dd t
    =
    \lim_{x\to 0} \nth x\intzi h_c(c-x,t) e^{-q t}\,\dd t.
  \end{align*}
  However, by scaling \eqref{e:h-scaling} and Proposition \ref{p:LT},
  for $0<x<c/2$,
  \begin{align*}
    \nth x\intzi h_c(c-x,t) e^{-q t}\,\dd t
    =
    \frac{e^{-q^\ra x}-1}{x}
    \sumoi n \frac{q^{n-1} c^{\alpha n-1}}{\Gamma(\alpha n)} +
    \sumoi n \frac{q^{n-1} [c^{\alpha n - 1} -(c-x)^{\alpha
    n-1}]}{x\Gamma(\alpha n)}.
  \end{align*}
  As a result,
  \begin{align*}
    \intzi m(c,t) e^{-q t}\,\dd t
    =
    \sumoi n \frac{q^{n-1} c^{\alpha n - 2}}{\Gamma(\alpha n-1)} - 
    q^{\ra}
    \sumoi n \frac{q^{n-1} c^{\alpha n-1}}{\Gamma(\alpha n)}.
  \end{align*}
  Provided that $\beta > q^\ra$, integration term by term of the \rhs
  yields
  \begin{align*}
    \intzi \Grp{\intzi m(c,t) e^{-q t}\,\dd t} e^{-\beta c}\,\dd c
    =
    \sumoi n\frac{q^{n-1}}{\beta^{\alpha n-1}}
    -
    q^\ra \sumoi n\frac{q^{n-1}}{\beta^{\alpha n}}
    = \frac{\beta - q^\ra}{\beta^\alpha - q}.
  \end{align*}
  By analytic extension, the equality still holds for $0\le \beta\le
  q^\ra$.   Then by \eqref{e:ladder0} the proof is complete.
\end{proof}

\begin{proof}[Proof of Lemma \ref{l:hc2}]
  By \eqref{e:pre-passage2} and integral by parts,
  \begin{align} \label{e:h1-ibp}
    h_1(1-x,t)
    =
    g_t(1-x) - g_t(1) + \int^t_0 \rsup F_{-x}(t-s) \frac{\partial
      g_s(1)}{\partial s}\,\dd s,
  \end{align}
  where $\rsup F_{-x}(t) = \int^\infty_t f_{-x}(s)\,\dd s =
  \pr{\fht_{-x} > t}$.  For $0<x<1/2$, $g_t(1-x) - g_t(1) = - g'_t(z)
  x$ for some $z\in (1-x,1)$.  Clearly $z>1/2$.  It is not hard to
  show that $M_1 := \sup_{y>0} [y^{\alpha + 2} |g'_1(y)|]<\infty$
  (\cite{sato:99:cup}, p.~88).  On the other hand, by $g_t(z) =
  t^{-\ra} g_1(t^{-\ra} z)$, $g'_t(z) = t^{-2\da} g'_1(t^{-\ra} z)$.
  Then
  \begin{align} \label{e:h1-ibp1}
    |g_t(1-x) - g_t(1)| = x t^{-2\da} |g'_1(t^{-\ra} z)|
    \le x t^{-2\da} M_1 (t^{-\ra} z)^{-\alpha -2} \le M_1 2^{\alpha +
    2} x t.
  \end{align}
  Next, by $g_s(1) = s^{-\ra} g_1(s^{-\ra})$, $|\partial
  g_s(1)/\partial s| \le (\ra) [s^{-\ra-1} g_1(s^{-\ra}) + s^{-2\da-1}
  |g'_1(s^{-\ra})|]$ is bounded.  Then for some $M_2>0$,
  \begin{align*}
    \Abs{\int^t_0 \rsup F_{-x}(t-s) \frac{\partial
      g_s(1)}{\partial s}\,\dd s
    }
    \le
    M_2 \int^t_0 \rsup F_{-x}(s)\,\dd s
    =
    M_2 x^\alpha \int^{x^{-\alpha} t}_0 \rsup F_{-1}(s) \,\dd s,
  \end{align*}
  where the equality is due to $\rsup F_{-x}(s) = \rsup
  F_{-1}(x^{-\alpha} s)$ and change of variable.  Because $\rsup
  F_{-1}(s)$ is decreasing with $\rsup F_{-1}(0)=1$ and is slowly
  varying at $\infty$ with index $-\ra$, there is a constant $M_3>0$
  such that $\int^y_0 \rsup F_{-1}(s)\,\dd s \le M_3 y^{1-\ra}$ for
  all $y>0$.  It follows that
  \begin{align} \label{e:h1-ibp2}
    \Abs{\int^t_0 \rsup F_{-x}(t-s) \frac{\partial
      g_s(1)}{\partial s}\,\dd s
    }
    \le M_2 M_3 x t^{1-\ra}.
  \end{align}
  Then the proof is complete by combining
  \eqref{e:h1-ibp}--\eqref{e:h1-ibp2}.
\end{proof}

\begin{proof}[Proof of Eq.~\eqref{e:coexcessive}]
  Denote the \rhs of \eqref{e:coexcessive} by $v^q(x)$.  The task is
  to show $\dual {v^q} = \dual {u^q}$, where, for example, $\dual
  {v^q}(x) = {v^q}(-x)$.  Since $v^q$ is a version of the
  $q$-resolvent density, according to the proof of Proposition I.13 of
  \cite{bertoin:96:cup},  $(r-q) U^r \dual {v^q}\uto \dual {u^q}$ as
  $r\toi$, where $U^r$ is the $r$-resolvent operator.  For $r>q$,
  \begin{align*}
    U^r \dual {v^q}(x)
    &=
    \intzi e^{-r t} \mean^x[\dual {v^q}(X_t)]\,\dd t
    \\
    &=
    \intzi e^{-r t} \Sbr{\int \Grp{\intzi e^{-q s} g_s(-y)\,\dd s}
      g_t(y-x)\,\dd y}\,\dd t 
    \\
    &=
    \intzi \intzi e^{-r t -q s} g_{s+t}(-x)\,\dd s\,\dd t
    \\
    &=
    (r-q)^{-1} \intzi (1-e^{(q-r) s}) e^{-q s}g_s(-x)\,\dd s.
  \end{align*}
  Then by monotone convergence, $(r-q) U^r \dual {v^q}(x)\to \dual
  {v^q}(x)$, giving $\dual {v^q}(x) = \dual {u^q}(x)$.
\end{proof}

\end{document}